\documentclass[12pt,a4paper]{amsart}

\usepackage[utf8]{inputenc}
\usepackage[english]{babel}

\usepackage{latexsym}
\usepackage{amsfonts}
\usepackage{amsmath}
\usepackage{amssymb}
\usepackage{graphicx}
\usepackage{color}
\usepackage{enumitem}

\usepackage{amsthm}
\usepackage{pstricks}
\usepackage{csquotes}
\usepackage{fullpage}
\usepackage{url}
\usepackage[ocgcolorlinks, linkcolor=red]{hyperref}

\usepackage{algorithm}
\usepackage{algcompatible}

\newcommand{\R}{\mathbb{R}}
\newcommand{\N}{\mathbb{N}}

\newcommand{\C}{\mathbb{C}}

\newcommand{\ov}[1]{\overline{#1}}
\newcommand{\eps}{\epsilon}

\newcommand\supp{\operatorname{supp}} 
 
\newcommand\spec{\operatorname{Spec}}

\newcommand\p{\partial}

\newcommand\spn{\operatorname{span}} 
\newcommand\Ker{\operatorname{Ker}} 
\newcommand\osupp{\operatorname{osupp}}

\newtheorem{thm}{Theorem}[section]
\newtheorem{cor}[thm]{Corollary}
\newtheorem{lem}[thm]{Lemma}
\newtheorem{prop}[thm]{Proposition}

\newtheorem{rem}[thm]{Remark}

\theoremstyle{definition}

\numberwithin{equation}{section}

\begin{document}

\title[]{The monotonicity method for inclusion detection and the time harmonic elastic wave equation}

\date{}

\author[]{Sarah Eberle-Blick}
\address{Institute of Mathematics, Goethe University Frankfurt, Germany}
\email{eberle@math.uni-frankfurt.de}

\author[]{Valter Pohjola}
\address{Institute of Mathematics, Goethe University Frankfurt, Germany}
\email{valter.pohjola@gmail.com}

\begin{abstract}
We consider the problem of reconstructing inhomogeneities in 
an isotropic elastic body using time harmonic waves.
Here we extend the so called  monotonicity  method 
for inclusion detection and show how to determine   certain types of inhomogeneities 
in the Lamé parameters and the density.
We also included some numerical tests of the method.
\end{abstract}

\maketitle

\tableofcontents

\section{Introduction}

\noindent
We study the problem of reconstructing the shape of inhomogeneities in elastic bodies 
from fixed frequency oscillations.
This problem is of importance in
nondestructive testing, and has potential applications in engineering, geoscience and medical imaging.
Our main aim is to formulate a monotonicity based reconstruction method.

The elastic properties of an isotropic elastic body $\Omega \subset \R^3$ can in the linear regime
be described by means of the Lamé parameters $\lambda$ and $\mu$.
The effect of a time harmonic oscillation in such a body $\Omega$ is described by
the Navier equation, which gives the behaviour of the displacement field 
$u : \Omega \to \R^3$, $u \in H^1(\Omega)^3$ of the solid body $\Omega$.
Here we consider $\Omega \subset \R^3$ with Lipschitz boundary,
and the Navier equation in terms of the following boundary 
value problem
\begin{align}  \label{eq_bvp1}
\begin{cases}
\nabla \cdot (\C\,  \hat \nabla u )  + \omega^2\rho u &= 0, \quad\text{in}\,\,\Omega,\\
\hspace{1.8cm}(\C\,  \hat \nabla u ) \nu  &= g, \quad\text{on}\,\,\Gamma_N, \\
\hphantom{\hspace{1.8cm}(\C\,  \hat \nabla u ) } u  &= 0, \quad\text{on}\,\,\Gamma_D, 
\end{cases}
\end{align}
where $\Gamma_N$ and $\Gamma_D$ are such that
$$
\Gamma_N, \Gamma_D \subset \p\Omega \text{ are open }, \qquad \Gamma_N \neq \emptyset, \qquad
\p \Omega = \ov{\Gamma}_N \cup \ov{\Gamma}_D,
$$
and where $\hat \nabla u  = \frac{1}{2}(\nabla u + (\nabla u)^T)$ is the symmetrization of the Jacobian
(or the strain tensor), 
and $\C$ is the 4th order tensor defined by
$$
(\C A)_{ij} = 2\mu A_{ij} + \lambda \operatorname{tr}(A) \delta_{ij},
\quad \text{ where } A \in \R^{3 \times 3},
$$
and $\delta_{ij}$ is the Kronecker delta. Here
$\lambda,\mu \in L^\infty_+(\Omega)$ are scalar functions called the Lamé parameters, that describe the elastic  properties
of the material, $\rho \in L^\infty_+(\Omega)$ is the density of the material,
and $\omega \neq 0$ the angular frequency of the oscillation, and $\nu$
is the outward pointing unit normal vector to the boundary $\p \Omega$. 
See sections \ref{sec_prelim} and \ref{sec_wellposed} for more details and definitions.
The vector field $g \in L^2(\Gamma_N)^3$ acts as the source
of the oscillation, and since $\C \,\hat \nabla u$ equals by Hooke's law to the Cauchy stress tensor, we see 
that the boundary condition $g$ specifies the traction on the surface $\p \Omega$.

We also make the standing assumption that $\omega \in \R$ is not a resonance frequency. By this we mean that
zero is not an eigenvalue for the mixed eigenvalue value problem \eqref{eq_Neumann}. When 
this assumption holds, the problem \eqref{eq_bvp1} admits a unique solution
for given boundary condition $g\in L^2(\Gamma_N)^3$. See Corollary \ref{cor_ZeroEigenvalue}. We can thus define the Neumann-to-Dirichlet map
$\Lambda:L^2(\Gamma_N)^3 \to L^2(\Gamma_N)^3$, as
$$
\Lambda: g \mapsto u|_{\Gamma_N}. 
$$
Thus $\Lambda$  maps the traction to the displacement $u|_{\Gamma_N}$ on the boundary.

\medskip
\noindent
In this paper we deal with the shape reconstruction problem, which is also known as the 
inclusion detection problem.
We are more specifically interested in
determining the region where the material properties $\lambda$, $\mu$ or $\rho$ differ from
some known constant background values $\lambda_0$, $\mu_0$ or $\rho_0$, when given 
a set of boundary measurements in the form of a Neumann-to-Dirichlet map $\Lambda$.
This problem corresponds physically to determining the inhomogeneous regions in the body $\Omega$ from 
boundary measurements that give the displacement due to  some tractions applied at the boundary.
Here we more specifically assume that $\lambda_0,\mu_0,\rho_0 > 0$ are some known constants,
and that there is a jump in the material parameters
\begin{align*} 
\lambda &= \lambda_0 + \chi_{D_1} \psi_1 ,\quad D_1 \Subset  \Omega, \\ 
\mu &= \mu_0 + \chi_{D_2} \psi_{2}, \quad D_2 \Subset  \Omega, \\ 
\rho &= \rho_0 - \chi_{D_3} \psi_{3}, \quad D_3 \Subset  \Omega, 
\end{align*}
where $\chi_{D_j}$ are the characteristic functions of the sets $D_j$ and
$\psi_j |_{D_j} \in L_+^\infty(D_j)$, so that there is a jump in the material parameters at the 
boundaries $\p D_j$ of the regions where the material parameters differ from the background values.

We study this problem using the so called monotonicity method. The monotonicity method 
is a shape reconstruction method first formulated in  \cite{TR02,Ta06} and also \cite{Ge08},
for the conductivity equation.
In \cite{HPS19b} the method was extended to the analysis of time harmonic waves in the context
of a Helmholtz type equation.
In \cite{EH21} the monotonicity method was formulated for the recovery of inhomogeneities
in an isotropic elastic body in the stationary case. In addition, 
the monotonicity methods can be applied for realistic data as shown, e.g., for the stationary elastic problem in \cite{EM21}.
Our aim here is to study this problem in the non-stationary, or time harmonic case, of the
Navier equation.
The work here thus combines the ideas of  \cite{HPS19b} and \cite{EH21}.

\medskip
\noindent
Our main result is the justification of the  shape reconstruction procedure outlined 
in Algorithm \ref{alg_shapeInclusion}. This is  done in Theorems \ref{thm_inclusionDetection_rho_mu} and 
\ref{thm_inclusionDetection_EVbound}. 
To understand the algorithm  consider the set $D \subset \Omega$ given by
$$
D = \supp(\lambda - \lambda_0) \cup \supp(\mu-\mu_0) \cup \supp(\rho-\rho_0).
$$
The set $D$ is the region where the material parameters differ from background.
Algorithm \ref{alg_shapeInclusion} approximates the set $\osupp(D)$, which stands for the outer support of the set 
$D$\footnote{We extend the terminology of \cite{HU13} to sets.}.
See definition \eqref{eq_def_osupp}. 
The set $\osupp(D)$ coincides with $D$ if $\Omega \setminus D$ is connected, and if not, then $\osupp(D)$
corresponds intuitively to $D$  and the union of all its internal cavities. 

Algorithm \ref{alg_shapeInclusion} generates a collection of subsets $\mathcal A$, 
such that $\cup \mathcal{A}$ approximates the set $\osupp(D)$. This 
gives a good approximation of the shape of the inhomogeneous region $D$, 
disregarding any internal cavities.
Algorithm \ref{alg_shapeInclusion} works roughly by choosing a collection of subsets
$\mathcal{B} = \{ B \subset \Omega\}$, and building an approximating collection $\mathcal{A}$ 
by choosing $B$, such that $B \subset \osupp(D)$. An approximation  of $\osupp(D)$ is then obtained
as $\cup \mathcal{A}$, from which we have removed any internal cavities.
Note that in Algorithm \ref{alg_shapeInclusion} $\Lambda$ denotes 
the measured Neumann-to-Dirichlet map, and $\Lambda^\flat$ a test Neumann-to-Dirichlet map
built using the set $B$ and the prescripts in Theorems
\ref{thm_inclusionDetection_EVbound} or \ref{thm_inclusionDetection_rho_mu}.

\begin{algorithm} \label{alg_shapeInclusion}
\caption{Reconstruction of the shape of an inhomogeneity $\osupp(D) \subset \Omega$}
\begin{algorithmic}[1]

\STATE  Choose a collection of sets $\mathcal{B} = \{ B \subset \Omega\}$.

\STATE  Set the approximating collection $\mathcal{A} = \{\}$.

\STATE  Compute $M_0$ using Theorem  \ref{thm_inclusionDetection_EVbound}.

\FOR{  $B \in \mathcal{B}$}

\FOR{  $\Lambda^\flat$ with parameters varied as suggested by Theorem \ref{thm_inclusionDetection_EVbound}}

\STATE Compute $N_B := \sum_{\sigma_k < 0} 1$, where $\sigma_k$ are the eigenvalues of $\Lambda - \Lambda^\flat$ 
\STATE (where the eigenvalues are counted as many times as its multiplicity indicates).

\IF {$ N_B \leq M_0 $}

\STATE 
Add $B$ to the approximating collection $\mathcal{A}$, since 
by Theorem \ref{thm_inclusionDetection_EVbound} we have 
\STATE 
that $B \subset \osupp(D)$. 

\ELSE

\STATE Discard $B$, since by Theorem \ref{thm_inclusionDetection_EVbound} $B \not \subset D_j$, $j=1,2,3$.

\ENDIF
\ENDFOR
\ENDFOR

\STATE  Compute the union of all elements in $\mathcal{A}$ and all components of 
$\Omega \setminus \cup \mathcal{A}$ not connected 
\STATE to $\p\Omega$. The resulting set is an approximation of $\osupp(D)$.

\end{algorithmic}
\end{algorithm}

\noindent
The main novelty in Algorithm \ref{alg_shapeInclusion} is that it extends the reconstruction
method for  Helmholtz type equations formulated in \cite{HPS19b} to the Navier equation,
thus generalizing the work in \cite{HPS19b} and \cite{EH21}.
Monotonicity methods for shape reconstruction have in general various advantages.
Some of the benefits of the monotonicity methods are that it
is simple to formulate and implement, 
it can be implemented efficiently using a linearization of Neumann-to-Dirichlet maps,
and that indefinite inhomogeneities can be recovered.

The shape reconstruction  and  related reconstruction methods 
for the Navier equation have received a fair amount of attention. This is partly 
due to the importance of this problem in geophysical and engineering applications.
Several methods have been used to analyze the shape reconstruction problem 
for the Navier equation and related equations. For the Navier  equation we have  
e.g. \cite{HKS12,HLZ13,EH19,GK08}, which are based on the factorization method.
For work based on iterative and other  methods  see  e.g. \cite{SFHC14,BYZ19,BHQ13,BC05} 
and  the references therein.
There are also a number of works related to the shape reconstruction problem in elasticity
in the stationary case $\omega =0$. See e.g. \cite{EH21,II08,Ik99}.
The shape reconstruction problem has also been extensively studied for scalar equations,
in particular  for electrical  impedance tomography, 
and acoustic and electromagnetic inverse problems. Several methods have been formulated. 
For this see e.g. \cite{Ta06,Ge08,GK08,Ik90,HPS19b,CDGH20}. 
Shape reconstruction using the monotonicity method for time harmonic waves have in particular
been explored by \cite{HPS19b,HPS19a,GH18,Fu20a,Fu20b,AG23}.

Calderón type inverse problems for isotropic linear elasticity  
have more generally been the subject of a number of studies. 
Early work on this problem include  \cite{Ik90} that deals with the linearized problem, and
\cite{ANS91} which deals with determining the coefficients at the boundary.
One of the first more comprehensive 
uniqueness  results  in the three dimensional case was obtained in \cite{NU94} and also in \cite{ER02}, where it is  shown that
the Lamé parameters $\mu$ and $\lambda$ are uniquely determined, in the stationary case by the Dirichlet-to-Neumann map
assuming that $\mu$ is close to a constant.
It should be noted that the general uniqueness problem is still open.
There are a number of earlier and subsequent works related to the Calderón type inverse problem.

We also present a preliminary numerical implementation of Algorithm \ref{alg_shapeInclusion}
in section \ref{sec_numerics}. 
We want to remark that Algorithm \ref{alg_shapeInclusion} is in  many ways an 
idealized formulation, and that it holds for a sufficiently fine discretization.
Since our computations do not have enough precision, we introduce a new
discrete bound $\tilde{M}$, which is motivated due to our numerical experiments
(see section \ref{sec_numerics}).
All in all, we end up with first numerical results for the elasto-oscillatory
case which shows us that in principle the monotonicity tests work.

We will now comment on the proofs and the main ideas that they involve in more detail. 
The analysis of the monotonicity method for time harmonic wave phenomena
in \cite{HPS19b} is largely based on spectral theoretic considerations.
Here we thus need at various points to fairly carefully analyze the 
spectral properties of
both the direct problem \eqref{eq_bvp1} in the form of \eqref{eq_Neumann},
and the Neumann-to-Dirichlet map as compact operator on $L^2(\Gamma_N)^3$.

One of the main ingredients of the
monotonicity method is the use of localized solutions. This was originally achieved
in \cite{Ge08} by the use of a range inclusion argument, which 
was also used to analyze the stationary case in \cite{EH21}.
Here we explore an alternative approach of localizing solutions using  
Runge approximation as in \cite{HPS19b}.
Runge approximation was originally formulated in the context of elliptic PDEs  
in \cite{La56} (for more on Runge approximation in the current context, see 
e.g. \cite{La56,HPS19b,HLL18,Po22}).
The Runge approach to localizing solutions is in principle straight forward,
see e.g. \cite{HLL18}.  
In the monotonicity method for time harmonic waves we however need 
to work modulo a finite dimensional subspace as in Lemma \ref{lem_monotonicity_ineq1} and in \cite{HPS19b}.
To create localized solutions we  need to approximate certain candidate solutions on subsets 
of $D_1,D_2\subset\Omega$, with desired properties.
For this we have to find enough candidate solutions with non-zero divergence. 
One central problem is for us thus to find a suitable criterion on a boundary condition of the Navier equation
that guarantees non-zero divergence. 
This is one of the main challenges in making our Runge argument work,
 and this analysis is contained in subsection \ref{sec_div}. The main idea here is to
relate the Navier equation to the time harmonic Maxwell equation, and utilize the theory
related to Maxwell's equations.
This connection can also be used to create solutions when $\mu$ and $\rho$ are constant, 
with no divergence.
The localized solutions also need to  have a non-zero strain tensor $\hat \nabla u$.
Achieving this is  simpler, and for this we use a variant of the Korn inequality, 
see equation \eqref{eq_korn2_2} and Lemma \ref{lem_symGradient}.

A major obstacle in recovering multiple coefficients from a single frequency $\omega$ 
is that uniqueness is expected to fail  when a density term is added to the stationary equation.
One expects there to be several distinct tuples of coefficients 
that match the given boundary measurements, in analogy with the scalar case where this happens.
A straight forward mechanism for obtaining counter examples to uniqueness in the scalar case
of optical tomography was given in \cite{AL98}, where it was
shown that one cannot recover both a diffusion and an absorption coefficient
from boundary measurements. 
This phenomenon also holds more widely, and thus a  similar form of non-uniqueness should  effect  
the corresponding inverse problem for the Navier equation.
We note in this context that our shape reconstruction method works in the  general case
only for inhomogeneities where the density either remains constant or becomes 
smaller than the background density. 
One should also note that that two different frequencies are enough to recover multiple coefficients
in the closely related acoustic problem, see \cite{Na88}.
One can also impose special conditions on the coefficients, such as them being
piecewise constant or analytic 
in order to recover several coefficients from a single frequency, see  \cite{Ha09}.
Piecewise constant coefficients are  in the case of the   Navier equation  explored
in \cite{BHFVZ17}, where  stability and  uniqueness and local reconstruction
are considered.

Several aspects of inverse problems are left unstudied in the current paper.
We only consider the specific form of perturbations, where the Lamé parameters increases and the density
decreases in the inhomogeneous regions, or where the Lamé parameters remain constant 
and the density increases. Natural questions  which we do not address here 
is if the case where all parameters increase in the homogeneity can somehow be handled,
and if we can get some improvements when setting some of the coefficients as constant.
Our numerical implementation is also in some aspects preliminary.  We do not 
consider a regularized method that could handle noisy data nor a linearized version
that would be more efficient.

\medskip
\noindent
The paper is structured as follows. In section \ref{sec_prelim} we review some definitions and properties 
that will be used in the rest of the paper. In section \ref{sec_direct_spec} we study the direct problem and the mixed 
eigenvalue problem related to \eqref{eq_bvp1}. Section \ref{sec_mono_ineq} contains the 
monotonicity inequality that underlies our analysis. In section \ref{sec_localization}
we create the localized solutions that become small or large in certain prescribed subsets.
In section \ref{sec_shape} we justify Algorithm \ref{alg_shapeInclusion}.
In the last section, we conduct some numerical tests of the algorithm.

\section{Preliminaries}\label{sec_prelim}

\noindent
In this section we will review some notations, definitions, and preliminary results that are used throughout the 
paper.
We begin by reviewing the definitions related to function spaces. We define
$$
L^\infty_+(\Omega) := \big\{ f \in L^\infty(\Omega) \;:\; \operatorname{essinf}_\Omega f > 0 \big\}.
$$
The space $H^1(\Omega)$ denotes the $L^2(\Omega)$ based Sobolev space with one weak derivative.
If $X$ is a function space, then $Z^n := Z \times \dots \times Z$, where the right hand side 
contains $n$ copies of $Z$.
The $L^2$-inner product is denoted by $(\cdot,\cdot)_{L^2}$, so that 
$$
(u,v)_{L^2(\Omega)^n} := \int_\Omega u \cdot v\,dx, \quad u,v \in L^2(\Omega)^n.
$$
We use the notation $\perp$ for orthogonality with respect to the $L^2$-inner product, unless otherwise stated,
so that 
$$
u \perp v \qquad \Leftrightarrow \qquad (u,v)_{L^2(\Omega)^n} = 0, \quad  \text{ when } u,v \in L^2(\Omega)^n.
$$
Next we will review some definitions related to matrices and vectors.
The Frobenius inner product $A:B$ is defined as
$$
A:B = \sum_{ij} A_{ij}B_{ij},  \qquad A,B \in \R^{m\times n}.
$$
We denote the Euclidean norm on $\R^{m \times n}$, $m,n \in \N$, by
$$
|A| = (A:A)^{1/2},\qquad \text{ when } A \in \R^{m \times n}.
$$
Note that this definition applies also to vectors and scalars.
The matrix divergence is give by
$$
(\nabla \cdot A)_{i} =  (\nabla \cdot (\operatorname{row}_i A))_i, \qquad A \in H^1(\Omega)^{3 \times 3}. 
$$
We will use a tensor form of the divergence theorem, that states
$$ 
\int_\Omega \nabla \cdot  A  u \,dx = - \int_\Omega A:\hat \nabla u \,dx + \int_{\p \Omega} A \nu \cdot u\,dS,
$$
for $A \in H^1(\Omega)^{n\times n}$. See \cite{Ci88} p. xxix. 
Note that this gives the tensor form of the divergence theorem, since
$A$ can here be  interpreted  as a 2nd order tensor, provided that the coefficients are regular enough.
The bilinear form related to equation \eqref{eq_bvp1} is given 
by
\begin{align}  \label{eq_weak}
B(u,v)  := -\int_\Omega 
2 \mu \hat \nabla u :\hat \nabla v + \lambda \nabla \cdot u \nabla \cdot v - \omega^2\rho u\cdot v\,dx, 
\end{align}
for all $u,v \in H^1(\Omega)^3$. For the precise definition of a weak solution that we employ here 
see section \ref{sec_wellposed}.
We will also use the notations
\begin{align} \label{eq_L}
L_{\lambda,\mu,\rho} u :=  \nabla \cdot (\C \hat \nabla u )  + \omega^2\rho u
= \nabla \cdot ( 2 \mu \hat \nabla u + \lambda (\nabla \cdot u) I ) + \omega^2\rho u,
\end{align}
where $I$ is the identity matrix. 

When the Lamé parameters and density are regular, and $u$ solves \eqref{eq_bvp1} with $g$, and
$v$ solves \eqref{eq_bvp1} with $h$,
we see by integrating by parts that
\begin{align}  \label{eq_NDmap}
B(u,v)  = -\int_{ \Gamma_N} g \cdot v \,dS = -( g, \Lambda h)_{L^2(\Gamma_N)^3}. 
\end{align}
As it will later turn out, this holds also for a weak solution $u \in H^1(\Omega)^3$ by definition \eqref{eq_weak2}.
Let $\Gamma \subset \p \Omega$ be open. We abbreviate the  boundary condition in \eqref{eq_bvp1} by
\begin{align*} 
\gamma_{\mathbb{C},\Gamma} u  = (\C \, \hat \nabla u ) \nu |_{\Gamma},
\end{align*}
or with $\gamma_\mathbb{\C} u $ if the boundary is clear from the context. 
Note that these notations are formal when $u \in H^1(\Omega)^3$, since we cannot in general take the trace of a 
$L^2(\Omega)^3$ function. In the low regularity case we understand the boundary condition in a weak sense
as suggested by \eqref{eq_weak2}. 
We can  also define $\gamma_{\mathbb{C}} u \in L^2(\Gamma_N)^3$, 
provided that $u \in H^1(\Omega)^3$ solves \eqref{eq_bvp1}. In this case we use duality,
so that
\begin{align*} 
-(\gamma_{\mathbb{C}} u, \, \varphi|_{\Gamma_N}  )_{L^2(\Gamma_N)^3}
= B(u,\varphi), \qquad \forall \varphi \in H^1(\Omega)^3. 
\end{align*}

\noindent
When the boundary value problem \eqref{eq_bvp1} admits a unique solution for 
all $g \in L^2(\Gamma_N)^3$, then the Neumann-to-Dirichlet map $\Lambda$  is well-defined.
\\
\\
The localized solutions we consider in section \ref{sec_localization}, are built using the 
unique continuation principle, which is hence also an important tool in the sequel.
We need a version of the unique continuation principle that applies to the Navier equation.
For this we use Theorem 1.2 in \cite{LNUW11}.
Notice that these results require the Lamé parameter
$\mu$ to be Lipschitz continuous.

\begin{prop} \label{prop_UCP} Assume that $\Omega$ is open and connected.
Let $\lambda,\rho \in L^\infty(\Omega)$ and $\mu \in W^{1,\infty}(\Omega)$, be such that
\begin{align*} 
&\mu(x),\; 2\lambda(x) + \mu(x) > \delta_0,\quad \forall \text{ a.e. } x \in \Omega, \\
&\| \mu \|_{ W^{1,\infty}(\Omega) } + \| \lambda \|_{ L^\infty(\Omega) } \leq M_0,\quad  
\|\rho \|_{ L^\infty(\Omega)  } \leq M_0,
\end{align*}
where $M_0,\delta_0 > 0$.
Assume that $u$, solves the equation $L_{\lambda,\mu,\rho} u = 0$ in $\Omega$ and $u = 0$
in some open subset $U \subset \Omega$, then $u=0$ in $\Omega$.
\end{prop}

\begin{proof}
This is a direct consequence of the strong unique continuation principle 
of Theorem 1.2 in \cite{LNUW11}.
\end{proof}

\begin{prop} \label{prop_bndryUCP} Assume that $\p \Omega$ is $C^{1,1}$ smooth and that
$\lambda,\rho \in L^\infty(\Omega)$ and $\mu \in W^{1,\infty}(\Omega)$ are as in Proposition 
\ref{prop_UCP}. Suppose $u$ is such that
$$
L_{\lambda,\mu,\rho} u = 0, \quad \text{ in } \Omega,  \qquad u|_\Gamma = 0,  \qquad \gamma_\C u |_\Gamma = 0,
$$
where $\Gamma \subset \p \Omega$ is open and non-empty, then $u=0$ in $\Omega$.
\end{prop}

\begin{proof}
We can show that the claim follows from Proposition \ref{prop_UCP}. If $x_0 \in \Gamma$, then
we can extend the solution $u$ by zero, to $B(x_0,r) \setminus \Omega$, where $B(x_0,r)$ is a ball,
with sufficiently small $r>0$.
The claim follows now  from Proposition \ref{prop_UCP}.
\end{proof}

\section{The direct problem and  spectral properties} \label{sec_direct_spec}

\noindent
In the next two subsections we will investigate the existence and uniqueness of
the solutions to \eqref{eq_bvp1} and a related mixed eigenvalue problem in some detail. 
We furthermore derive a norm estimate where the norm of the strain tensor controls the norm of the Jacobian, 
and use spectral representation to derive a convenient formula for the expression $B(w,w)$. 
These results will be of importance in the subsequent sections.

\subsection{A variational argument} \label{sec_wellposed}
In this subsection we investigate the well-posedness of the inhomogeneous problem
\eqref{eq_weak2}, and further derive a useful norm estimate.
In the next subsection we use these results in conjunction with Fredholm theory
to show that the boundary value problem \eqref{eq_bvp1} is well-posed, and investigate 
the related eigenvalue problem.
Through out this subsection we  assume that $\Omega \subset \R^3$ is a 
bounded Lipschitz domain and that  $\lambda,\mu,\rho \in L_+^\infty(\Omega)$.
Recall moreover that we assume that 
$$
\Gamma_N, \Gamma_D \subset \p\Omega \text{ are open }, \qquad \Gamma_N \neq \emptyset, \qquad
\p \Omega = \ov{\Gamma}_N \cup \ov{\Gamma}_D.
$$

We will first show that there exists a unique 
weak solution $u \in H^1(\Omega)^3$ to the mixed boundary value problem 
\begin{align}  \label{eq_bvp2}
\begin{cases}
\nabla \cdot (\C\,  \hat \nabla u )  + \omega^2\rho u + \tau u&= F, \\
\;\quad\quad\quad\quad\quad\quad(\gamma_\C u ) |_{ \Gamma_N} &= g, \\	
\;\quad\quad\quad\quad\quad\quad\quad\quad u |_{\Gamma_D} &= 0, 
\end{cases}
\end{align}
for suitable $\tau \in \R$, when $g \in L^2(\Gamma_N)^3$ and $F \in L^2(\Omega)^{3}$. 

In order to specify what we mean by a weak solution to \eqref{eq_bvp2}, we define the bilinear form 
$$
B_\tau(u,v) := B(u,v) + \tau (u,v)_{L^2(\Omega)^3}, \quad u,v \in \mathcal{V},
$$
where $\mathcal{V} \subset H^1(\Omega)^3$ is the closed subspace
$$
\mathcal{V} := \{ u \in H^1(\Omega)^3 \,:\, u|_{\Gamma_D} = 0 \}.
$$
The weak solution $u \in \mathcal{V}$ to \eqref{eq_bvp2} is a vector field that satisfies
\begin{align}  \label{eq_weak2}
B_\tau(u,v)  = - \int_{\Gamma_N} g\cdot v \,dS + \int_\Omega F \cdot v \,dx
\end{align}
for every 
$v \in \mathcal{V}$.\footnote{Note that it easy to check that the existence of a weak solution in this sense,
implies the existence of a strong solution with the desired boundary conditions, when given enough regularity.}

One of the main tools for obtaining existence and uniqueness results in linear elasticity 
is the second Korn inequality. We will need the following version of this inequality
\begin{align}  \label{eq_korn2_1}
\| u \|_{H^1(\Omega)^3} \leq C \big(\| \hat \nabla u \|_{L^2(\Omega)^{3 \times 3}} 
+ \| u \|_{ L^2(\Omega)^3}\big), \quad u \in H^1(\Omega)^3, 
\end{align}
where $C >0$ is a constant.
A proof of this can be found in \cite{OSY92}, see Theorem 2.4 p.17.
Next we prove existence and uniqueness for solutions to the boundary value problem
\eqref{eq_bvp2}.

\begin{prop} \label{prop_Wellposedness}
There exists a $\tau_0 \leq 0$, for which
the boundary value problem in \eqref{eq_bvp2} admits a unique weak solution
$u \in H^1(\Omega)^3$, that satisfies
\begin{align}  \label{eq_aprioriEst}
\| u \|_{ H^1(\Omega)^3} \leq C \big(\| g \|_{ L^2(\Gamma_N)^3}+ \| F \|_{ L^2(\Omega)^3 } \big).
\end{align}
\end{prop}

\begin{proof}
We use the Lax-Milgram Lemma to prove the uniqueness and existence of a weak solution
(see e.g. Theorem 1.3 in \cite{OSY92}).
We need to show that the bilinear form $B_\tau: \mathcal{V}\times \mathcal{V}\to \R$ 
is coercive and continuous for some $\tau \leq 0$. For coercivity we need 
to show that 
\begin{align} \label{eq_Btau_coer}
|B_{\tau_0}(u,u)| \geq c \| u \|^2_{ H^1(\Omega)^3}, \quad c>0,
\end{align}
for some $\tau_0 \leq 0$. Let $\tau \leq 0$.
The Korn inequality \eqref{eq_korn2_1} gives that
\begin{align*} 
|B_\tau(u,u)| &\geq 
\int_\Omega 2 \mu \hat \nabla u :\hat \nabla u + \lambda \nabla \cdot u \nabla \cdot u 
-\tau \int_\Omega u^2\,dx
- \omega^2 \| \rho \|_{L^\infty}\int_\Omega u^2\,dx \\ 
&\geq
C \big(\| \hat \nabla u \|^2_{L^2(\Omega)^{3\times 3}} 
+\| \nabla \cdot u \|^2_{L^2(\Omega)} \big)
-(\tau + \omega^2\| \rho \|_{L^\infty})\| u \|^2_{L^2(\Omega)^3} \\
&\geq
C \| u \|^2_{H^1(\Omega)^3} 
-(C+\tau + \omega^2\| \rho \|_{L^\infty})\| u \|^2_{L^2(\Omega)^3}. 
\end{align*}
Thus if we choose a $|\tau|$ large enough with $\tau\leq 0$, the 2nd term is positive and can be dropped.
We thus see that we can choose a $\tau_0 \leq 0$  so that \eqref{eq_Btau_coer} holds.

The continuity of $B_\tau$ follows from the estimate
\begin{align*} 
|B_\tau(u,v)| &=  
\Bigg|\int_\Omega
2 \mu \hat \nabla u :\hat \nabla v + \lambda \nabla \cdot u \nabla \cdot v 
- (\tau + \omega^2\rho) u\cdot v\,dx
\Bigg| \\
&\leq
C (\| \hat \nabla u \|_{L^2(\Omega)^{3\times 3}} \| \hat \nabla v \|_{L^2(\Omega)^{3\times 3}}
+\| \nabla \cdot u \|_{L^2(\Omega)} \| \nabla \cdot v \|_{L^2(\Omega)}
+\| u \|_{L^2(\Omega)^3} \| v \|_{L^2(\Omega)^3}) \\
&\leq
C \| u \|_{H^1(\Omega)^3} \| v \|_{ H^1(\Omega)^3}.
\end{align*}
The last step before applying the Lax-Milgram Lemma is to check that 
the left side of \eqref{eq_weak} gives a continuous  functional on $\mathcal{V}$.
To this end note that by the Cauchy-Schwarz and trace inequalities  
\begin{align} \label{eq_rhs_cont}
- \int_{\Gamma_N } g\cdot v \,dS \leq 
\big\| g \,\big\|_{ L^2(\Gamma_N)^3 } \big\| v|_{\p \Omega} \big\|_{ L^2(\p\Omega)^3}
\leq C \| v \|_{ H^1(\Omega)^3},
\end{align}
and similarly by the Cauchy-Schwarz inequality 
\begin{align} \label{eq_rhs_cont_2}
\int_\Omega F \cdot v \,dx  \leq C \| v \|_{ H^1(\Omega)^3}.
\end{align}
The Lax-Milgram Lemma gives us thus the existence of a unique $u \in \mathcal{V} \subset H^1(\Omega)^3$
for which \eqref{eq_weak2} holds for all $v \in \mathcal{V}$, and when $\tau \leq 0$, $|\tau|$ is large enough.
We have thus found a weak solution in accordance with \eqref{eq_weak2}.

The final step is to verify the estimate of the claim. By 
\eqref{eq_weak2}, \eqref{eq_Btau_coer}, \eqref{eq_rhs_cont} and \eqref{eq_rhs_cont_2}
we have that
$$
\| u \|_{H^1(\Omega)^3 } \leq C (\| g \,\|_{ L^2(\Gamma_N)^3 }
+ \| F \|_{ L^2(\Omega)^{3}}).
$$
\end{proof}
\noindent
Later we will be interested in solutions, where the strain tensor $\hat \nabla u$ 
controls the entire Jacobian $\nabla u$. In order to investigate this we need to introduce 
some additional concepts. 
The space of rigid motions $\mathcal{R}$ is given by
$$
\mathcal{R} :=\big \{ \eta : \R^3 \to \R^3 \,:\, \eta(x) = Ax  + c,\; A \in \R^{3\times 3} 
\text{ is antisymmetric },
 c \in \R^3\big \}.
$$
(See e.g.  p.19 in  \cite{OSY92}.) With this subspace we can formulate
another version of the second Korn inequality that applies to
$u \in W$, where $W \subset H^1(\Omega)^3$ is a closed subspace and $ W\cap\mathcal{R} = \{0\}$.
For such $u$ we have that
\begin{align}  \label{eq_korn2_2}
\| u \|_{H^1(\Omega)^3} \leq C \| \hat \nabla u \|_{L^2(\Omega)^{3\times 3}}, 
\end{align}
where $C>0$ is a constant. For a proof of this see \cite{OSY92} theorem 2.5 p.19.
We can use \eqref{eq_korn2_2} to prove the following Lemma,
which gives a criterion on the boundary condition $g$ that allows for controlling 
the $H^1$-norm of the solution by the strain tensor, and which we use  in section \ref{sec_localization}.

\begin{lem} \label{lem_symGradient} Suppose $u \in H^1(\Omega)^3$ solves  \eqref{eq_bvp2}
with $F=0$, $\Gamma_D = \emptyset$, $\Gamma_N = \p\Omega$ and let $g \in L^2(\Omega)^3$ be such that
$$
g \perp \mathcal{R},\qquad \text{ in  the $L^2(\p \Omega)^3$-inner product}.
$$
then we have the estimate
\begin{align}  \label{eq_JacobianEst}
\| u \|_{H^1(\Omega)^3} \leq C \| \hat \nabla u \|_{L^2(\Omega)^{3\times 3}}, 
\end{align}
where $C>0$ is a constant.
\end{lem}

\begin{proof}
According to \eqref{eq_korn2_2} it is enough to show that
$$
\mathcal{W} \cap \mathcal{R} = \{0\},
$$
where $\mathcal{W} \subset H^1(\Omega)^3$ is the closed subspace given by
$$
\mathcal{W} := \big\{ u \in H^1(\Omega)^3 \;:\; u \text{ solves \eqref{eq_bvp2} with } F=0,
g \perp \mathcal{R} \big \}. 
$$
It is straight forward to check that this is indeed a closed subspace.

To check that the condition holds, assume that $u \in \mathcal{W} \cap \mathcal{R}$.
Then $u = \tilde r$, $\tilde r \in \mathcal{R}$.
Since $u$ is a weak solution we have by \eqref{eq_weak2} that
$$
B(\tilde r,r) = (g,r)_{L^2(\p\Omega)^3}=0, \qquad \forall r\in \mathcal{R}.  
$$
Choosing $r = \tilde r$ and
since $\tilde r = Ax +c$ where $A$ is antisymmetric, we have that
$\hat \nabla \tilde r=0$ and $\nabla \cdot \tilde r = 0$, and thus
$$
B(\tilde r ,\tilde r ) = k^2\int_\Omega \rho \tilde r^2  \,dx.
$$
Combining the two previous equations and using the fact that $\rho$ is strictly positive,
gives that  $u = \tilde r = 0$, which is what we wanted to show.
\end{proof}

\subsection{The spectrum, well-posedness and a representation formula}
Next we will use Fredholm theory to describe the eigenvalues and eigenfunctions
related to the boundary value problem \eqref{eq_bvp2}. This is a mixed eigenvalue
problem sometimes called the Zaremba eigenvalue problem, see \cite{Za10}. Note that if $\Gamma_D = \emptyset$,
then we are dealing with the Neumann eigenvalue problem.
In Corollary \ref{cor_ZeroEigenvalue} we show that the boundary value problem \eqref{eq_bvp1} 
is well-posed. At the end of this subsection we will also derive a spectral representation of the expression $B(w,w)$.

The next Proposition encapsulates the main spectral theoretic facts related to the direct problem 
that are needed in the sequel.

\begin{prop} \label{prop_Neumann_spectrum}
There exists a sequence of $\sigma_k \in \R$, such that
$$
\sigma_1 \geq \sigma_2 \geq  \sigma_3 \geq \dots \to - \infty
$$
and such that the mixed eigenvalue problem
\begin{align}  \label{eq_Neumann}
\begin{cases}
\nabla \cdot (\C \, \hat \nabla \varphi_k )  + \omega^2\rho\varphi_k &= \sigma_k \varphi_k,\\
\quad\quad\quad\quad\quad\quad \gamma_\mathbb{C} \varphi_k  |_{\Gamma_N} &= 0, \\
\quad\quad\quad\quad\quad\quad\quad \varphi_k  |_{\Gamma_D} &= 0,
\end{cases}
\end{align}
has a  non-zero solution $\varphi_k \in H^1(\Omega)^3$. We have moreover that:
\begin{enumerate}	
\item The eigenvalues $\{\sigma_k\}$ are of finite multiplicity and  the eigenfunctions $\{\varphi_k\}$ 
form an orthonormal basis in $L^2(\Omega)^3$. \\
\item There are only finitely many positive eigenvalues $\{\sigma_1,..,\sigma_K \}$. \\
\item The boundary  value problem 
\begin{align}  \label{eq_bvpLsource}
\begin{cases}
\nabla \cdot (\C\,  \hat \nabla u )  + \omega^2\rho u + \tau u&= F, \\ 
\;\quad\quad\quad\quad\quad\quad\quad \gamma_\C u  |_{\Gamma_N} &= 0, \\	
\quad\quad\quad\quad\quad\quad\quad\quad u  |_{\Gamma_D} &= 0,
\end{cases}
\end{align}
admits a 
unique solution $u \in H^1(\Omega)^3$ when $-\tau \notin \{ \sigma_1,\sigma_2,\dots \}$,
and $F \in L^2(\Omega)^3$. \\
\item 
Suppose $0 \in \{ \sigma_1,\sigma_2,\dots \}$. Then there exists 
a unique solution $u \in H^1(\Omega)^3 / \mathcal{N}_0$\footnote{This means the quotient space.}
to the boundary  value problem \eqref{eq_bvpLsource} when $\tau=0$, provided that 
$$
(F,\psi)_{L^2(\Omega)^3} = 0,\qquad \forall \psi \in \mathcal{N}_0,
$$
where $\mathcal{N}_0$ is the eigenspace corresponding to zero. 
\end{enumerate}
\end{prop}

\begin{proof} 
Pick $g=0$ in \eqref{eq_bvp2}, and let $\tau_0$ be as in 
Proposition \ref{prop_Wellposedness}.
Then Proposition \ref{prop_Wellposedness} gives the right inverse operator
$$
 K_{\tau_0} := (L_{\lambda,\mu,\rho} + \tau_0)^{-1} : L^2(\Omega)^3 \to H^1(\Omega)^3,
$$
that maps a source term $F$ to the corresponding weak solution.
Note that this is a compact operator on $L^2(\Omega)^3$, since the inclusion 
$H^1(\Omega)^3 \subset L^2(\Omega)^3$ is compact.

$K_{\tau_0}$ is also self-adjoint, since $K_{\tau_0}$ is bounded and since
\begin{align*} 
(K_{\tau_0}F,G)_{L^2(\Omega)^3} = B_{\tau_0}(K_{\tau_0}F,K_{\tau_0}G) =  B_{\tau_0}(K_{\tau_0}G,K_{\tau_0}F) = (F,K_{\tau_0}G)_{L^2(\Omega)^3}. 
\end{align*}
The spectral Theorem for compact self-adjoint operators, see \cite{RSI} Theorem VI.16, shows
that the spectrum of $K_{\tau_0}$ consists of eigenvalues $\tilde \sigma_k \in \R$, such that
\begin{align}  \label{eq_neg}
K_{\tau_0} \tilde \varphi_k = \tilde \sigma_k \tilde \varphi_k, \qquad \tilde \sigma_k \to 0, \text{ as } k \to \infty,
\end{align}
and  that the corresponding eigenfunctions $\tilde \varphi_k$ form an orthonormal basis of $L^2(\Omega)^3$.
Notice also that all the eigenvalues are negative. Using the fact that $K_{\tau_0}$ is the right inverse and orthonormality,
we have that 
\begin{small}
\begin{align}  \label{eq_sign}
C\| \tilde \varphi_k \|^2_{H^1(\Omega)^3}
\leq -B_{\tau_0}(\tilde \varphi_k, \tilde \varphi_k) 
=\leq -\tfrac{1}{\tilde \sigma_k} B_{\tau_0}(K_{\tau_0}\tilde \varphi_k, \tilde \varphi_k)  
= -\tfrac{1}{\tilde \sigma_k} (\tilde \varphi_k,\,\tilde \varphi_k)_{L^2(\Omega)^3} = -\tfrac{1}{\tilde \sigma_k}.
\end{align}
\end{small}
To obtain a solution to the eigenvalue problem \eqref{eq_Neumann}, notice that
$$
\tilde \sigma_k B_{\tau_0}(\tilde \varphi_k, w) = B_{\tau_0}(K_{\tau_0} \tilde \varphi_k, w) = (\tilde \varphi_k,\,w)_{L^2(\Omega)^3}.
$$
So if we define $\sigma_k := \tilde \sigma_k^{-1} - \tau_0$ and $\varphi_k := \tilde \varphi_k$, then we have that
$$
B(\varphi_k,w) = B_{\tau_0}(\tilde \varphi_k, w) - \tau_0 (\tilde \varphi_k,w)_{L^2(\Omega)^3}
= \sigma_k (\varphi_k,w)_{L^2(\Omega)^3}.
$$
Thus $\varphi_k$ and $\sigma_k$ solve the weak form of the eigenvalue problem
and by \eqref{eq_sign} we see that the eigenvalues converge to $-\infty$. This proves the first
part of the claim.
Part $(1)$ follows from the fact that $\tilde \varphi_k$ 
is an orthonormal basis in $L^2(\Omega)^3$. The second part $(2)$ follows now from 
\eqref{eq_neg} and \eqref{eq_sign}. 

\medskip
\noindent
We will use the Fredholm alternative to prove claim $(3)$. 
Let $F \in L^2(\Omega)^3$, $F\neq0$. 
The existence of a weak solution $u$ to \eqref{eq_Neumann} means that
\begin{align}  \label{eq_weakB}
B_\tau(u,v) = (F,v)_{L^2(\Omega)^3}, \text{ for all } v \in \mathcal{V}. 
\end{align}
Since $B_{\tau_0}(u,v) = (F+(\tau_0-\tau) u ,v)_{L^2(\Omega)^3} $ if and only if $ u = K_{\tau_0}(F + (\tau_0-\tau) u)$,
we have that
\begin{align} \label{eq_intEq}
u - (\tau_0-\tau) K_{\tau_0} u = K_{\tau_0} F.
\end{align}
By the Fredholm alternative for a compact operator $T$ on $L^2(\Omega^3)$ we know that 
either
\begin{align} \label{eq_nonHomogeneous}
\exists ! u \in L^2(\Omega)^3 , \quad   u - T u = h, \quad \forall h \in L^2(\Omega)^3, 
\end{align}
or the homogeneous equation has  a non trivial solution, i.e.
\begin{align} \label{eq_Homogeneous}
\exists u \in L^2(\Omega)^3, u \neq 0,\quad u - T u = 0.
\end{align}
See p. 203 in \cite{RSI}.
Thus by the dichotomy between \eqref{eq_nonHomogeneous} and  \eqref{eq_Homogeneous}
applied to $T = (\tau_0-\tau)K_{\tau_0}$,
we see that \eqref{eq_intEq} has a unique solution  if and only if 
$$
1 \notin \spec((\tau_0 - \tau)K_{\tau_0}) \quad\Leftrightarrow\quad 
\frac{1}{\tau_0 -\tau}  \neq \tilde \sigma_k
\quad\Leftrightarrow\quad -\tau \neq \frac{1}{\tilde \sigma_k} - \tau_0 = \sigma_k. 
$$
It thus follows that \eqref{eq_weakB} has a unique weak solution $u$ 
if and only if $-\tau \notin \{\sigma_1,\sigma_2,\sigma_3,..\}$.

\medskip
The last step is to prove (4). We will again use Fredholm theory. Note firstly that now $\tau_0 \neq 0$, 
since $0$ is an eigenvalue of the boundary value problem.
Since we assume that $\sigma_k = 0$, for some $k$, we have that $\tilde \sigma_k = \tau_0^{-1}$
is an eigennvalue of $K_{\tau_0}$. 
And thus there is a non trivial $w\neq 0$, solving
\begin{align} \label{eq_intEq_2}
w - \tau_0 K_{\tau_0} w = 0. 
\end{align}
The Fredholm alternative of Theorem 2.27 in \cite{Mc00} p.37--38, implies that 
the eigenspace of $K_{\tau_0}$ corresponding to $\tilde \sigma_k = \tau_0^{-1}$, which we denote by
$\mathcal{M}_k$ and to which $w$ belongs, is such that
$$
\dim\big(\mathcal{M}_k\big) < \infty.
$$
Since $K_{\tau_0}$ is self-adjoint, Theorem 2.27 in \cite{Mc00} gives furthermore that 
\begin{align}  \label{eq_Kh}
u - \tau_0 K_{\tau_0} u =  h 
\end{align}
is solvable if and only if
$$
(h,\psi)_{L^2(\Omega)^3} = 0, \qquad \forall \psi \in \mathcal{M}_k.
$$
Now since $K_{\tau_0}$ is self-adjoint, we have for $\psi \in \mathcal{M}_k$  that
\begin{align*} 
(F,\psi)_{L^2(\Omega)^3} = 0
\quad\Leftrightarrow\quad
(K_{\tau_0}F,\psi)_{L^2(\Omega)^3} = 0,
\end{align*}
so that \eqref{eq_Kh} is solvable for $h = K_{\tau_0}F$, if $F \perp \mathcal{M}_k$.
It follows that there exists a weak solution $u$ to
\begin{align}  \label{eq_weakB_2}
B_0(u,v) = (F,v)_{L^2(\Omega)^3}, \text{ for all } v \in \mathcal{V}, 
\end{align}
when $F \perp \mathcal{M}_k$. It is easy to check that this solution is
unique in $H^1(\Omega)^3 / \mathcal{M}_k$. By definition we also have 
$\mathcal{M}_k = \mathcal{N}_0$ where $\mathcal{N}_0$ denotes the eigenspace
of zero for the problem \eqref{eq_Neumann}.
\end{proof}

\noindent As a corollary to the previous Lemma we get the well-posedness 
of the boundary value problem of \eqref{eq_bvp1}, assuming that $\omega$ is not
a resonance frequency.

\begin{cor} \label{cor_ZeroEigenvalue}
Let $\{\sigma_1,\sigma_2,..\}$ be the eigenvalues given by Proposition \ref{prop_Neumann_spectrum}.
For the boundary value problem
\begin{align}  \label{eq_bvp_inh}
\begin{cases}
\nabla \cdot (\C\,  \hat \nabla u )  + \omega^2\rho u &= 0, \\
\;\quad\quad\quad\quad\quad \gamma_\C u  |_{\Gamma_N} &= g, \\	
\;\quad\quad\quad\quad\quad \quad u  |_{\Gamma_D} &= 0, 
\end{cases}
\end{align}
with $g \in L^2(\Gamma_N)^3$, the following holds:
\begin{enumerate}
\item The problem \eqref{eq_bvp_inh} admits a unique solution if $0 \notin \{ \sigma_1,\sigma_2,..\}$.
\item If zero is an eigenvalue, i.e. $0 \in \{ \sigma_1,\sigma_2,..\}$,
then \eqref{eq_bvp_inh} admits a unique solution $u \in H^1(\Omega)^3 \,/\mathcal{N}_0$ for all $g\in L^2(\Gamma_N)^3$
that satisfy
$$
(g,\psi)_{L^2(\Gamma_N)^ 3} = 0,\quad \forall \psi \in \mathcal{N}_0,
$$
where $\mathcal{N}_0$ is the eigenspace of zero. 
\end{enumerate}
\end{cor}

\begin{proof}
Using Proposition \ref{prop_Wellposedness} we can define an inverse trace operator
$E_{\tau_0}:L^2(\Gamma_N)^3 \to H^1(\Omega)^3$ as
$$
E_{\tau_0}g := v,
$$
where $v$ is the unique solution to \eqref{eq_bvp2}, with $F=0$,
and where $\tau_0$ is as in Proposition \ref{prop_Wellposedness}.

Let us prove part (2) of the claim. The proof of (1) follows similar lines.
We can  prove (2) by reducing  it to finding a solution $w$ to the inhomogeneous case 
\begin{align}  \label{eq_bvp_inh_2}
\begin{cases}
L_{\lambda,\mu,\rho} w &= \widetilde F, \\
\quad \gamma_\C w  |_{\Gamma_N} &= 0, \\	
\quad \quad w  |_{\Gamma_D} &= 0, 
\end{cases}
\end{align}
and applying Proposition \ref{prop_Neumann_spectrum}.
To solve \eqref{eq_bvp_inh}
it is sufficient  to solve \eqref{eq_bvp_inh_2} with 
$$
\widetilde F:=-L_{\lambda,\mu,\rho}E_{\tau_0}g = \tau_0 v,
$$
where the last equality holds, since $E_{\tau_0}g=v$ solves by definition $(L_{_{\lambda,\mu,\rho}} + \tau_0) E_{\tau_0}g = 0$.
Note firstly that since $v \in H^1(\Omega)^3$, also $\widetilde F \in H^1(\Omega)^3$.
By Proposition \ref{prop_Neumann_spectrum} part (4) we know that \eqref{eq_bvp_inh_2} admits a unique solution
in $H^1(\Omega)^3/\mathcal{N}_0$.  

To complete the proof, we need to check that we can 
rewrite the condition appearing in part (4) of Proposition \ref{prop_Neumann_spectrum}, in the form appearing in (2) of the claim. 
To this end note that since a $\psi \in \mathcal{N}_0$  is a weak solution to the  
mixed eigenvalue problem for zero,
and since $v$ is a solution to \eqref{eq_bvp2} with $g$ as Neumann data and $\tau = \tau_0$, we have that
\begin{small}
\begin{align*} 
0 = (\widetilde F , \psi)_{L^2(\Omega)^3} =
\int_\Omega \tau_0 v \cdot \psi \,dx = \int_{\Omega} -\C \hat \nabla \psi : \hat \nabla v  + \omega^2 \rho\psi \cdot v \,dx + \int_\Omega \tau_0 v\cdot  \psi \,dx
= -\int_{\Gamma_N} g \cdot \psi \,dS.
\end{align*}
\end{small}
\noindent
We thus have the condition
\begin{align} \label{eq_f_ortho}
0 =  (g,\psi)_{L^2(\Gamma_N)^ 3}.
\end{align}
And hence we see that \eqref{eq_bvp_inh_2} is solvable when this condition is satisfied.
The solution to \eqref{eq_bvp_inh} is then given by
$$
u = w - E_{\tau_0} g.
$$
We thus see that $u$ satisfies \eqref{eq_bvp_inh} when $g$ satisfies \eqref{eq_f_ortho},
and consequently that (2) holds.
\end{proof}

\noindent
We now turn to the final matter in this subsection.
The eigenvalues $\{ \sigma_k\}$ and eigenfunctions $\{ \varphi_k\}$ allow us to 
'diagonalize' the differential operator $L_{\lambda,\mu,\rho}$ 
in \eqref{eq_bvp1}. This gives us the  identity 
$$
(L_{\lambda,\mu,\rho}w,w)_{L^2(\Omega)^3} = \sum_k \sigma_k c_k^2, \quad \text{ where }  \quad w = \sum_k c_k \varphi_k,
$$
for any  $w \in \mathcal{V}$. 
In the rest of this subsection we will derive some results that gives us a weak version of this 
that holds when $\lambda,\mu \in L_+^\infty(\Omega)$. 

\medskip
\noindent By the coercivity and continuity of the bilinear form $B_{\tau_0}$, for suitable ${\tau_0} \leq 0$,
which are guaranteed to exist by Proposition \ref{prop_Wellposedness}, we have 
that
$$
(u,v)_{B_{\tau_0}} := -B_{\tau_0}(u,v), \qquad u,v \in \mathcal{V}, 
$$
gives an inner product on $\mathcal{V}$, and a corresponding norm $\| u \|^2_{ B_{\tau_0} } := (u,u)_{B_{\tau_0}}$. 
This allows us to extend the result of Proposition \ref{prop_Neumann_spectrum}.

\begin{cor} \label{cor_onorm}
Suppose $\varphi_k$ and $\sigma_k$ are the eigenfunctions given by Proposition \ref{prop_Neumann_spectrum}
and let ${\tau_0} \leq 0$ be as in Proposition \ref{prop_Wellposedness}.
Then we have that
$$
\psi_k := \frac{\varphi_k}{[-(\sigma_k + {\tau_0})]^{1/2}} , \qquad k=1,2,..
$$
form an orthonormal basis of $\mathcal{V}$ in the inner product $(\cdot, \cdot)_{B_{\tau_0}}$ and corresponding norm.	
\end{cor}

\begin{proof} To prove that the set $\{\psi_k \}$ is orthonormal, notice that
\begin{align*} 
(\psi_k,\psi_j)_{B_{\tau_0}} = -B(\psi_k,\psi_j) - {\tau_0}(\psi_k,\psi_j)_{L^2(\Omega)^3} 
=
\frac{-B(\varphi_k,\varphi_j) - {\tau_0}(\varphi_k,\varphi_j)_{L^2(\Omega)^3}}{[-(\sigma_k + {\tau_0})]^{1/2}[-(\sigma_j + {\tau_0})]^{1/2}} = 
\delta_{jk},
\end{align*}
where we used the weak definition of the eigenvalue problem and the orthogonality of the functions $\varphi_k$.

It remains to check that the set is complete. Suppose that $u \in \mathcal{V}$. Completeness follows if we can show that
the condition 
$(u,\psi_k)_{B_{\tau_0}} = 0$, for every $k$, implies that $u = 0$. To this end we note that if
$$
0 = (u,\psi_k)_{B_{\tau_0}} =  \frac{(u,\varphi_k)_{B_{\tau_0}}}{[-(\sigma_k + {\tau_0})]^{1/2}} 
$$
for every $k$, then
$$
0 = (u,\varphi_k)_{B_{\tau_0}} = -B(u,\varphi_k) - {\tau_0}(u,\varphi_k)_{L^2(\Omega)^3}
= -(\sigma_k + {\tau_0})(u,\varphi_k)_{L^2(\Omega)^3}.
$$
We know that $(\sigma_k + {\tau_0}) < 0$, since ${\tau_0}$ makes $B_{\tau_0}$ coercive, so that
$(u,\varphi_k)_{L^2(\Omega)^3}=0$, for every $k$.
The functions  $\varphi_k$ are an orthonormal basis, and thus $u=0$.
\end{proof}

\begin{lem} \label{lem_diagonaliztion}
Suppose that $w \in \mathcal{V}$, then
$$
B(w,w) = \sum_k \sigma_k c_k^2,
$$
where $w=\sum_k c_k \varphi_k$ is the representation via the orthonormal basis $\lbrace\varphi_k\rbrace$ by Proposition \ref{prop_Neumann_spectrum}.
\end{lem}

\begin{proof}
According to Corollary \ref{cor_onorm} we have the representation 
$$
w=\sum_k c'_k \psi_k, \quad \psi_k := \frac{\varphi_k}{[-({\tau_0}+\sigma_k)]^{1/2}} , \quad c'_k := (w,\psi_k)_{B_{\tau_0}},
$$
where $(u,v)_{B_{\tau_0}} = -B_{\tau_0}(u,v)$ and where the series converges in the corresponding  $\|\cdot\|_{B_{\tau_0}}$-norm.
We now have that
\begin{align*} 
-B_{\tau_0}(w,w) 
=
-B_{{\tau_0}}\Big(\sum_k c'_k \psi_k ,\, \sum_j c'_j \psi_j\Big)
= \sum_k |c'_k|^2.
\end{align*}
On the other hand we have that
$$
-B_{\tau_0}(w,w) = -B(w,w) - {\tau_0} (w,w)_{L^2(\Omega)^3},
$$
so that 
$$
B(w,w) = -\sum_k |c'_k|^2 - {\tau_0} (w,w)_{L^2(\Omega)^3}.
$$
Let us compute the coefficients $c'_k$ in terms of $c_k$. We have that
\begin{align*} 
c'_k = (w,\psi_k)_{B_{\tau_0}} &= -B(w,\psi_k) - {\tau_0} (w,\psi_k)_{L^2(\Omega)^3} 
=\frac{-\sigma_k c_k - {\tau_0} c_k}{[-({\tau_0}+\sigma_k)]^{1/2}} 
= [-(\sigma_k + {\tau_0})]^{1/2} c_k.
\end{align*}
From the earlier equality we get that 
$$
B(w,w) = -\sum_k [-(\sigma_k + {\tau_0})] c^2_k  - {\tau_0} \sum_k c^2_k
= \sum_k \sigma_k c^2_k,
$$
which is what we wanted to prove.
\end{proof}

\section{Monotonicity inequalities} \label{sec_mono_ineq}

\noindent
In this section we derive some monotonicity relations that are of fundamental 
importance in justifying monotonicity based shape reconstruction, and will be needed
in the later sections. The use of this type of inequalities goes back to 
\cite{KSS97} and \cite{Ik98}.

It will be convenient to use the quantity  $d(\lambda,\mu,\rho)$, which we define  as
\begin{align}  \label{eq_def_d}
d(\lambda,\mu,\rho) := \text{ the number of $\sigma_k > 0$ in problem \eqref{eq_Neumann} counted with multiplicity.}
\end{align}
The following Lemma is the main monotonicity inequality we will be using. 

\begin{lem} \label{lem_monotonicity_ineq1}
Let $\mu_j,\lambda_j,\rho_j \in L^\infty_+(\Omega)$, for $j=1,2$ and $ \omega \neq 0$.
Let $u_j$ denote the solution to \eqref{eq_bvp1} where $\mu=\mu_j, \lambda= \lambda_j$ and $ \rho=\rho_j$,
with the boundary value  $g$. There exists a finite dimensional subspace $V\subset L^2( \Gamma_N)^3$, 
such that
\begin{align*}  
\big(  (\Lambda_2 - \Lambda_1)g, \,g  \big )_{L^2(\Gamma_N)^3} 
\geq
\int_\Omega 
2(\mu_1-\mu_2 ) |\hat \nabla u_1 |^2 + (\lambda_1 - \lambda_2 ) |\nabla \cdot u_1|^2  + \omega^2(\rho_2-\rho_1) |u_1|^2 \,dx, 
\end{align*}
when $g \in V^\perp$. We have moreover that $\dim(V) \leq d(\lambda_2,\mu_2,\rho_2)$.
\end{lem}

\begin{proof}
Let $B_j$ and $\Lambda_j$ be the bilinear form and Neumann-to-Dirichlet map given in \eqref{eq_weak}
and \eqref{eq_NDmap}, with $\mu=\mu_j, \lambda= \lambda_j$ and $ \rho=\rho_j$.
Notice firstly that from \eqref{eq_weak} we can deduce that
\begin{align*}  
B_1(u_1,u_1) = -\int_{\Gamma_N} g\cdot u_1 \,dS =  B_2(u_2,u_1). 
\end{align*}
In particular we have that 
$$
B_1(u_1,u_1) = - B_1(u_1,u_1) +2 B_2(u_2,u_1). 
$$
This and the symmetry of $B$, gives that
\begin{align*} 
\big(  (\Lambda_2 - \Lambda_1)g, \,g  \big )_{L^2(\Gamma_N)^3} 
&= 
B_1(u_1,u_1) - B_2(u_2,u_2)   \\
&= 
-B_1(u_1,u_1) +  2 B_2(u_2,u_1) - B_2(u_2,u_2)  \pm B_2(u_1,u_1)  \\
&= 
-B_1(u_1,u_1) + B_2(u_1,u_1) -  \big( B_2(u_2,u_2)  - 2 B_2(u_2,u_1)  + B_2(u_1,u_1) \big) \\ 
&= 
B_2(u_1,u_1) - B_1(u_1,u_1) - B_2(w,w),
\end{align*}
where $w = u_2-u_1$.
The claim follows if we can show that there exists a finite dimensional subspace $V$, such that
$$
B_2(w,w) \leq 0 \quad \text{ for } g \in V^\perp,
$$
since then we have that
\begin{align*} 
\big(  (\Lambda_2 - \Lambda_1)g, \,g  \big )_{L^2(\Gamma_N)^3} 
\geq
B_2(u_1,u_1) - B_1(u_1,u_1)  \quad \text{ for } g \in V^\perp,
\end{align*}
and since this is equivalent to the claim.

Let $\{\sigma_k\}$ and  $\{\varphi_k\}$ denote the eigenvalues and 
eigenfunctions of the mixed eigenvalue problem for the operator $L_{\mu_2, \lambda_2,\rho_2}$, that
are given by Proposition \ref{prop_Neumann_spectrum}.
The positive eigenvalues are finite, because of Proposition \ref{prop_Neumann_spectrum} part (2).
Suppose they are enumerated as
$$
\sigma_1 \geq \dots \geq \sigma_{K} \geq 0.
$$
Assume now that $w \perp \{\varphi_1,..,\varphi_K\}$ in the $L^2(\Gamma_N)^3$-inner product. Clearly we have that $w \in \mathcal{V}$. 
By the diagonalization result of Lemma \ref{lem_diagonaliztion}, we have, when writing $w = \sum_k c_k \varphi_k$, 
that
\begin{align*} 
B_2(w,w)  = \sum_{k > K} \sigma_k c_k^2 \leq 0,
\end{align*}
since $\sigma_k < 0$, when $k\geq K$.
The last step of the proof is to write the condition $w \in \spn\{ \varphi_1,..,\varphi_{K} \}^\perp =: W^\perp$
in terms of the boundary condition $g$. Let $S_j : L^2(\Gamma_N)^3 \to L^2(\Omega)^3$ be the solution
operators given by Corollary \ref{cor_ZeroEigenvalue}, for which
$$
S_j:g\mapsto u_j, \quad j=1,2.
$$
Now we have that
$$
\big(w,\varphi_k\big)_{L^2(\Omega)^3}
=\big((S_2-S_1)g,\varphi_k\big)_{L^2(\Omega)^3} 
=\big(g,(S_2-S_1)^*\varphi_k\big)_{L^2(\Gamma_N)^3}. 
$$
And thus that
$$
w = (S_2-S_1)g \in W^\perp \quad\Leftrightarrow\quad g \in ((S_2-S_1)^*W)^\perp.  
$$
Note that $\dim ((S_2-S_1)^*W) < \dim(W)$, and that $\dim(\mathcal{W}) = d(\lambda_2, \mu_2, \rho_2)$.
The claim of the lemma holds when we choose $V=(S_2-S_1)^*W$.

\end{proof}

\noindent
As a direct consequence of Lemma \ref{lem_monotonicity_ineq1} we have the following.

\begin{lem} \label{lem_LambdaMono}
Assume that $\mu_1 \geq \mu_2$, $\lambda_1 \geq \lambda_2$ and $\rho_2 \geq \rho_1$.
Then there exists a finite dimensional subspace $V \subset L^2(\Gamma_N)^3$ such
that
$$
(\Lambda_2 g,g)_{L^2(\Gamma_N)^3}
\geq 
(\Lambda_1 g,g)_{L^2(\Gamma_N)^3}, \qquad \forall g \in V^\perp.
$$
\end{lem}

\begin{proof}
Notice that Lemma \ref{lem_monotonicity_ineq1} gives directly that
\begin{small}
\begin{align*}  
\big(  (\Lambda_2 - \Lambda_1)g, \,g  \big )_{L^2(\Gamma_N)^3} 
\geq
\int_\Omega 
2(\mu_1-\mu_2 ) |\hat \nabla u_1|^2  + (\lambda_1 - \lambda_2 ) |\nabla \cdot u_1|^2  + \omega^2(\rho_2-\rho_1) |u_1|^2 \,dx \geq 0,
\end{align*}
\end{small}
for $g \in V^\perp$, which implies the claim.

\end{proof}

\section{Localization by Runge approximation}\label{sec_localization}

\noindent
We will prove localization results for solutions to the boundary value problem in 
\eqref{eq_bvp1}. The use of localized solutions go back to \cite{Ge08}, and plays
a fundamental part in justifying monotonicity based shape reconstruction.
Here we use an approach based on Runge approximation as in \cite{HPS19b}. 
The localized solutions will later be used in conjunction with the monotonicity inequality
of Lemma \ref{lem_monotonicity_ineq1}. The main challenge in this is to 
employ Runge approximation with solutions which on the boundary lie in $V^\perp$,
where $V$ is the subspace in Lemma \ref{lem_monotonicity_ineq1}. We deal with this
as in \cite{HPS19b}. Here we however have the extra complication of localizing
$\nabla \cdot u$ and $\hat \nabla u$ suitably.

We begin by showing  how to localize the solutions assuming we have the approximation 
results of Lemma \ref{lem_blowupCand} and Lemma \ref{lem_approxCrit} at hand. We will
postpone the proofs of Lemmas \ref{lem_blowupCand} and \ref{lem_approxCrit} to
subsections \ref{sec_div} and \ref{sec_runge}, as these need more a detailed analysis.

\medskip
\noindent
We will localize a solution so that it is small in some set $D_1 \subset \Omega$
and large in another set $D_2 \subset \Omega$. We will assume that $\p D_1$ is Lipschitz
and that $\p D_2$ is smooth, and moreover that
\begin{align}  \label{eq_Dassump}
D_1 \cap D_2 = \emptyset, \qquad \Omega \setminus (D_1 \cup D_2) \text{ is connected}, \qquad
\overline{\Omega \setminus (D_1 \cup D_2)}\cap\Gamma_N \neq \emptyset.
\end{align}

\begin{prop} \label{prop_localization1}
Assume that $D_1,D_2 \subset \Omega$ are  as in \eqref{eq_Dassump}.
Let $V\subset L^2(\Gamma_N)^3$ be a subspace with $\dim (V)<\infty$,
then there exists a sequence $g_j \in L^2(\Gamma_N)^3$ s.t. $ g_j \perp V$ in the $L^2(\Gamma_N)^3$-norm,
such that
$$
\| u_j \|_{L^2(D_1)^3  } \to 0, \qquad
\| u_j \|_{L^2(D_2)^3  } \to \infty,
$$
as $j \to \infty$, and where $u_j$ solves \eqref{eq_bvp1}, with the boundary conditions $g_j$.
\end{prop}

\begin{proof}
By Lemma \ref{lem_blowupCand} and Lemma \ref{lem_approxCrit} there is a function  
$$
w = 
\begin{cases}
 w \quad\text{ in }D_2 \quad\text{ with }\|  w \|_{ L^2(D_2)^3 } \neq 0,  \\
0 \quad\,\,\text{ in }D_1.
\end{cases}
$$
which can be approximated by solutions $w_j$ to \eqref{eq_bvp1} 
with the boundary  condition $ \tilde g_j  \in V^\perp$, such that 
$$
\|  w_j - w \|_{L^2(\Omega)^3 } \leq \tfrac{1}{j^2}, \quad j \in \N.
$$
We now set $u_j := jw_j$ and $g_j := j\tilde g_j$, then
$$
\|  u_j - jw \|_{L^2(\Omega)^3 } \leq \tfrac{1}{j}.
$$
It is easy to see that $u_j$ satisfy the desired norm estimates.
The claim follows from this.
\end{proof}

\noindent
The next Proposition extends Proposition \ref{prop_localization1}
by showing that we can also localize the derivatives in a suitable way.

\begin{prop} \label{prop_loc2}
Assume that $D_1,D_2 \subset \Omega$ are  as in \eqref{eq_Dassump}, and that $D'_i \Subset D_i$, $i=1,2$
are open and non-empty.
Let $V\subset L^2(\Gamma_N)^3$ be a subspace with $\dim (V)<\infty$,
then there exists a sequence $g_j \in L^2(\Gamma_N)^3$, such that $ g_j \perp V$ in the $L^2(\Gamma_N)^3$-norm,
and for which
$$
\| u_j \|_{ L^2(D_1)^3 }, 
\;\| \hat \nabla u_j \|_{ L^2(D'_1)^{3\times 3} }, 
\;\| \nabla \cdot u_j \|_{ L^2(D'_1)} 
\to 0, 
$$
and for which
$$
\| u_j \|_{ L^2(D_2)^3 }, 
\;\| \hat \nabla u_j \|_{ L^2(D'_2)^{3\times 3} }, 
\;\| \nabla \cdot u_j \|_{ L^2(D'_2)} 
\to \infty, 
$$
as $j \to \infty$, and where $u_j$ solves \eqref{eq_bvp1}, with the boundary conditions $g_j$.
\end{prop}

\begin{proof}
Let $u_j$ and $w$ be as in the proof of Proposition \ref{prop_localization1}.
By Lemma \ref{lem_blowupCand} we know that we can pick these, so that 
\begin{align*} 
\;\| \hat \nabla (jw) \|_{ L^2(D_2)^{3\times 3} }, 
\;\| \nabla \cdot (jw) \|_{ L^2(D_2)} 
\to \infty, 
\end{align*}
as $j \to \infty$, and  so that
\begin{align*} 
\;\| \hat \nabla (jw) \|_{ L^2(D_1)^{3\times 3} }, 
\;\| \nabla \cdot (jw) \|_{ L^2(D_1)} 
 = 0. 
\end{align*}
By elliptic regularity we have that
\begin{align*} 
\| \hat \nabla u_j - \hat \nabla (jw) \|_{L^2(D'_i)^{3\times 3}}
\leq
\|  u_j -  jw \|_{H^1(D'_i)^3}
\leq 
C \| u_j - jw \|_{L^2(D_i)^3} \to 0,
\end{align*}
as $j \to \infty$. 
It follows that 
$$
\;\| \hat \nabla u_j \|_{ L^2(D'_2)^{3\times 3}} \to \infty \qquad
\;\| \hat \nabla u_j \|_{ L^2(D'_1)^{3\times 3}} \to 0.
$$
The claim for $\nabla \cdot u_j$ can be proved in a similar manner.

\end{proof}
\noindent
To complete the proof of Propositions \ref{prop_localization1} and \ref{prop_loc2} we need to prove the lemmas
that they rely on. We will do this in the next two subsections.

\subsection{Boundary conditions for divergence} \label{sec_div}
In this subsection we start proving Lemmas  \ref{lem_approxCrit} and \ref{lem_blowupCand}
that were needed by Propositions \ref{prop_localization1} and \ref{prop_loc2}
in the previous part.
One of the main challenges in making the Runge approximation argument work in the next subsection is
to find a general enough criterion that guarantees that a solution to the Navier equation
has divergence. 
We  will look for a criterion on the boundary condition of the solution.
The complicating factor is that the criterion has to be such that it works even when we
consider boundary conditions in the orthogonal complement of some finite dimensional subspace.

The following Lemma gives a characterization of zero divergence solutions to the Navier equation.
The Lemma says essentially that a zero divergence solution to the Navier equation
is already uniquely determined by the tangential part of a Dirichlet condition 
(modulo some finite dimensional subspace, which appears if $\omega$ is a type of 
a resonance frequency).
The second lemma uses this result to formulate a condition of the Dirichlet
value of a solution that guarantees zero divergence.

The main idea behind the lemma is to relate a zero divergence solution of the Navier 
equation to a second order Maxwell type system  of the form
$$
\nabla \times (\nabla \times u) + \dots = 0,
$$
where the dots indicate lower order terms. The theory related to Maxwell's equations
tells us that the boundary condition $\nu \times u |_{\p D} =f$ is enough to determine the
solution $u$ uniquely. Note that this is 
less information than in a Dirichlet condition $u|_{\p\Omega} = f$
that is needed for the Navier equation, since the first condition 
does not determine the normal component of $u$.

\begin{lem} \label{lem_tang_unique} 
Assume that $\lambda,	\mu, \rho \in C^\infty(D) \cap L^\infty_+(\Omega)$,
where $D \subset \R^3$ is a bounded domain with smooth boundary.
Suppose $u \in H^1(D)^3$ solves
\begin{align}  \label{eq_bvp_tang}
\nabla \cdot (\C \hat \nabla u ) +  \omega^2 \rho u &= 0, \quad \text{ in } D.
\end{align}
And suppose that $u$ is such that 
$$
(u \times \nu)|_{\p D} = 0 \quad \text{and } \quad \nabla \cdot u = 0,\quad \text{ in } D.
$$
Then there exists a subspace $\mathcal{N} \subset H^1(D)^3$, with $\dim(\mathcal{N}) < \infty$,
and for which $u=0$ as an element in $H^1(D)^3/\mathcal{N}$.
\end{lem}

\begin{proof}
Assume that $u\in H^1(D)^3$ is a weak solution of \eqref{eq_bvp_tang} and $\nabla \cdot u = 0$.
Then
\begin{align} \label{eq_B0}
B(u,v) = 0 \qquad \forall v \in C^\infty_0(D)^3.
\end{align}
We can choose a sequence $u_k \in C^\infty(D)^3$, with $\nabla \cdot u_k = 0$, such that
$u_k \to u$, in the $H^1(D)^3$-norm as $k \to \infty$.
Integrating by parts gives that
\begin{align*} 
-B(u_k,v) =
\int_D 2 \mu \hat \nabla u_k :\hat \nabla v - \omega^2\rho u_k \cdot v\,dx 
=
\int_D  -\nabla \cdot ( 2 \mu \hat \nabla u_k ) \cdot v - \omega^2\rho u_k \cdot v\,dx. 
\end{align*}
We rewrite the 2nd order term. Notice 
first that by a direct computation we have 
that
$$
2 \nabla \cdot \hat \nabla u_k - \nabla \cdot ((\nabla \cdot u_k) I) = \Delta u_k. 
$$
Using this and that $\nabla \cdot u_k = 0$, we can add zero suitable so that
\begin{align*} 
\nabla \cdot ( 2 \mu \hat \nabla u_k ) 
= 2 \mu  \nabla \cdot \hat \nabla u_k +  2 \nabla \mu \hat \nabla u_k 
= \mu \Delta u_k +  2 \nabla \mu \hat \nabla u_k.
\end{align*}
From this and the identity $\Delta u = \nabla(\nabla \cdot u) - \nabla \times (\nabla \times u)$, we get that
\begin{align*} 
\nabla \cdot ( 2 \mu \hat \nabla u_k ) 
= -\mu \nabla \times (\nabla \times u_k) +  2 \nabla \mu \hat \nabla u_k.
\end{align*}
Using this and the integration by parts formula for the curl, we get that
\begin{align*} 
-B(u_k,v) 
&= \int_D (\mu \nabla \times (\nabla \times u_k) -  2 \nabla \mu \hat \nabla u_k)  \cdot v - \omega^2\rho u_k \cdot v\,dx \\
&= \int_D \nabla \times u_k \cdot \nabla \times (\mu v) -  2 \nabla  \tilde \mu \hat \nabla u_k  \cdot (\mu v) 
- \omega^2\tilde \rho u_k \cdot (\mu v)\,dx, 
\end{align*}
where we define for convenience $\tilde \mu := \log(\mu)$ and $\tilde \rho := \frac{\rho}{2\mu}$.
It is straight forward to see that, when
defining a bilinear form $B'$, and  taking the limit $k \to \infty$ of the previous equation
and using \eqref{eq_B0}, that 
\begin{align*} 
B'(u,\mu v) 
:= \int_D \nabla \times u \cdot \nabla \times (\mu v) -  2 \nabla  \tilde \mu \hat \nabla u   \cdot (\mu v) 
- \omega^2\tilde \rho u  \cdot (\mu v)\,dx = 0, 
\end{align*}
for all $v \in C^\infty_0(D)^3$.
By writing $w = \mu v$, for some $v \in C^\infty_0(D)^3$, we see that in fact
\begin{align}  \label{eq_Bprime}
B'(u,w) = 0, \qquad 
\forall w \in C^\infty_0(D)^3. 
\end{align}
Now consider the space
$$
H_0(\operatorname{curl}; \,D) := \big\{ u \in L^2(D)^3 \,:\, \nabla \times u \in L^2(D)^3,\, (\nu \times u)|_{\p D} = 0 \big\},
$$
equipped with the norm $\| v \|^2_{ H(\operatorname{curl};\,D) }:= \| \nabla \times v \|^2_{ L^2(D)^3 }+\| v  \|^2_{ L^2(D)^3 }$.
The set $C_0^\infty(\Omega)^3$  is dense in $H_0(\operatorname{curl}; \,D)$. Thus it follows from \eqref{eq_Bprime}
with straight forward estimates that
\begin{align*} 
B'(u,w) = 0, \qquad 
\forall w \in H_0(\operatorname{curl}; \,D)
\end{align*}
also holds. Consider the closed subspace $X \subset H^1(D)^3$, defined by
$$
X := \big\{ u \in H^1(D)^3 \,:\, \nabla \cdot u = 0,\, (\nu \times u)|_{\p D} = 0 \big\}.
$$
Firstly since $X \subset H_0(\operatorname{curl}; \,D)$ we also have that
\begin{align} \label{eq_Bhomogeneous}
B'(u,w) = 0, \qquad 
\forall w \in  X.
\end{align}
To show that $u$ is zero modulo some finite dimensional subspace we consider the inhomogeneous problem, 
related to \eqref{eq_Bhomogeneous}, of finding a  $u_F \in X$ for which
\begin{align}  \label{eq_weakApprox}
B'(u_F,v) =  \int_{D} F\cdot v \,dx, \quad \forall v \in X,
\end{align}
for a given $F \in L^2(D)^3$.\footnote{ Note that this corresponds to showing that a weak approximate solution $w$ in the subspace $X$ to
\begin{align*}  
\begin{cases}
\nabla \times (\nabla \times  w) - \nabla \tilde \mu \hat \nabla w -  \omega^2 \tilde \rho w &= 0, \\
\;\quad\quad\quad\quad\quad\quad\quad\quad\quad (\nu \times w ) |_{\p \Omega} &= 0, 
\end{cases}
\end{align*}
needs to be zero modulo some finite dimensional subspace.}
We will apply the Lax-Milgram Lemma. 
First we check that  $B'(u,u)+\tau(u,u)_{L^2}$, is coercive on $X$, for suitable $\tau \geq 0$.
To this end  note that by the triangle inequality, we get that
\begin{align*} 
|B' (v,v) |
&= 
\Big|\int_{D} \nabla \times v \cdot \nabla \times v \,dx
- \int_D \nabla \tilde \mu \hat \nabla v \cdot v - \omega^2 \tilde \rho v \cdot v \,dx\Big| \\
&\geq
\| \nabla \times v \|^2_{L^2(D)^3}
- \Big| \int_D \nabla \tilde \mu \hat \nabla u \cdot v \, dx \Big|  - C \| v \|^2_{ L^2(D)^3 }.
\end{align*}
To obtain a lower bound in the $H^1$-norm we will use Friedrich's or Gaffney inequality.
From this it follows that
\begin{equation} \label{eq_friedrichs}
\begin{aligned}
\| v \|_{ H^1(D)^3} 
&\sim \| \nabla \times v \|_{ L^2(D)^3 }+\| \nabla \cdot v  \|_{ L^2(D) }+\| v  \|_{ L^2(D)^3 } + \| \nu \times v \|_{H^{1/2}(\p D)^3} \\
&= \| \nabla \times v \|_{ L^2(D)^3 }+\| v  \|_{ L^2(D)^3 },
\end{aligned}
\end{equation}
since $v \in X$. See e.g. \cite{Ce96} Corollary 5 on p.51.  
From \eqref{eq_friedrichs} we get now  that 
\begin{align*} 
| B' (v,v) |
&\geq
\| v \|^2_{H^1(D)^3}
- \Big| \int_D \nabla \tilde \mu \hat \nabla v \cdot v \, dx \Big|  - (C+1) \| v \|^2_{ L^2(D)^3 }.
\end{align*}
To estimate the middle term notice that 
\begin{align*} 
\Big| \int_D \nabla \tilde \mu \hat \nabla v \cdot v \, dx \Big|  
\leq C \| v \|_{ H^1(D)^3 }\| v \|_{ L^2(D)^3} 
\leq \eps \| v \|^2_{ H^1(D)^3 } + \frac{C}{\eps} \| v \|^2_{ L^2(D)^3}, 
\end{align*}
with $\eps >0$. The two previous inequalities give the estimate
\begin{align*} 
| B' (v,v) |
&\geq
(1 - \eps) \| v \|^2_{H^1(D)^3} - \frac{C'}{\eps}  \| v \|^2_{ L^2(D)^3 },
\end{align*}
for small enough $\eps > 0$. By choosing a large enough $\tau \geq 0 $
we see that $B'(u,u) + \tau(u,u)_{L^2}$ is coercive on $X$.
It is straight forward to see that
\begin{align*} 
|B' (v,u) | \leq  \|u \|_{H^1(D)^3  } \|v \|_{H^1(D)^3}
\end{align*}
holds, and that $B'(u,u) + \tau(u,u)_{L^2}$ is also continuous. 

Let $\tau_0 \geq 0$ be such that $ B'_{\tau_0} := B' + \tau_0(\cdot,\cdot)_{L^2}$ is coercive and continuous. Then
the Lax-Milgram lemma  gives us a unique $u_F \in X$, such that
$$
B'_{\tau_0}(u_F,v) = (F,v)_{L^2(D)^3}, \qquad \forall v \in X.
$$
Thus there is a $\tau_0 \geq 0$ such that we may define the inverse operator $K'_{\tau_0}$, as
$$
K'_{\tau_0}: L^2(D)^3 \to L^2(D)^3,\qquad 
K'_ {\tau_0}F := u_F. 
$$
It is illustrative to first consider the case when $\tau_0 = 0$. 
Since $u$ solves \eqref{eq_Bhomogeneous}, we have 
thus by uniqueness $u=0$. In this case the claim of the lemma holds, with the choice $\mathcal{N} =\{0\}$.

\medskip
\noindent
Consider now the case where we cannot choose $\tau_0 = 0$. 
In this case we will generally have a non trivial eigenspace $\mathcal{N}$ corresponding to 
the eigenvalue zero of $K'_{\tau_0}$.
The operator $K'_{\tau_0}$ is compact on $X$ by Sobolev embedding, since $u_F \in H^1(D)^3$.
Since $u_F$ solves \eqref{eq_weakApprox} with $F=0$, we see that
$$
 B'(u_F,v) = 0,\qquad \forall v \in X.
$$
We will show that this implies that $u_F=0$ in $X / \mathcal{N}$, for some $\mathcal{N} \subset X$,
with $\dim(\mathcal{N}) < \infty$.
From the above equation we get that 
$$
 B'_{\tau_0}(u_F,v) = \tau_0(u_F,v),\qquad \forall v \in X,
$$
and therefore
\begin{align}  \label{eq_KKer}
u_F = \tau_0 K'_{\tau_0} u_F. 
\end{align}

We now use Fredholm theory.
By Theorem 2.22 in \cite{Mc00} p.35 we see that $I - \tau_0 K'_{\tau_0}$ is a Fredholm operator
of index $0$.
By Theorem 2.27 in p.37-38 in \cite{Mc00}, we have that for any operator $A:X \to Y$ of index $0$, 
we have that
$$
\dim (\Ker A) = p < \infty.
$$
Applying this to $A = I - \tau_0 K'_{\tau_0}$, gives that
$$
\dim(\Ker(I - \tau_0 K'_{\tau_0})) = p < \infty.
$$
By \eqref{eq_KKer} we have that $u_F \in \Ker\big(I - \tau_0 K'_{\tau_0}\big)$ so
$$
u_F = 0 \qquad\text{ in }\quad X / \mathcal{N},
$$
when we choose $\mathcal{N} := \Ker(I - \tau_0 K'_{\tau_0})$.

\end{proof}

\begin{rem} \label{rem_divSol}
If $\mu$ and $\rho$ are constant then it seems that the argument could 
with some modifications be used to construct solutions $u$
with $\nabla \cdot  u = 0$.
\end{rem}

\noindent
We use Lemma \ref{lem_tang_unique} to find a criteria for when a solution has divergence.
The essential content of the next Lemma is that a non-zero Dirichlet condition, 
with a tangential component that is zero, will always create a solution with non-zero divergence.

\begin{cor} \label{cor_divNonZero}
Assume that $\lambda,	\mu, \rho \in C^\infty(D)\cap L^\infty_+(\Omega)$, 
where $D \subset \R^3$ is a bounded domain with smooth boundary.
Suppose that $u \in H^1(D)^3$ solves
\begin{align}  \label{eq_bvp5}
\begin{cases}
\nabla \cdot (\C \hat \nabla u ) +  \omega^2 \rho u &= 0, \quad \text{ in } D\\
\;\quad\quad\quad \quad\quad u  |_{\p D} &= f, 
\end{cases}
\end{align}
where $f \in H^{1/2}(D)^3$ and $f \neq 0$. Then there exists a
finite dimensional subspace $\mathcal{N} \subset L^2(\p D)^3$ such that if 
$$
\nu \times f |_{\p D} = 0 \quad \text{ and }\quad f \perp \mathcal{N},
$$
then $\nabla \cdot u \neq 0$ in $D$.
\end{cor}

\begin{proof}
Let $\mathcal{N}'$ denote the subspace $\mathcal{N}$ in Lemma \ref{lem_tang_unique}.
Here we define $\mathcal{N} \subset L^2(\p D)^3$, as the subspace
$$
\mathcal{N} := \{ (\nu\cdot \psi|_{\p D})\nu \,:\, \psi \in \mathcal{N}' \}.
$$
Now suppose that the claim does not hold, i.e. $f\neq0$ and
$\nu \times f |_{\p D} = 0,$ and $ f \perp \mathcal{N}$
but $\nabla \cdot u \equiv 0$ in $D$.
By Lemma \ref{lem_tang_unique} we have that $u = 0$ in $H^1(D)^3 / \mathcal{N'}$, so that
we can express $u$ as
$$
u =  \sum_{k=1}^K c_k \psi_k,\qquad \psi_k \in \mathcal{N}'.
$$
We can decompose $u$ on $\p D$ as $u = (u\cdot \nu) \nu - \nu \times \nu \times u$.
This and the orthogonality condition on $f$ implies that
$$
(f , f)_{L^2(\p D)^3} =  (f , u|_{\p D})_{L^2(\p D)^3} 
= \big(f, (\nu \cdot u)\nu|_{\p D} \big)_{L^2(\p D)^3}=\sum_{k=1}^K c_k (f, \big(\nu \cdot \psi_k)\nu|_{\p D}\big)_{L^2(\p D)^3} = 0.
$$
This is a contradiction, since we assumed that $f \neq 0$.

\end{proof}

\subsection{A Runge type argument} \label{sec_runge}
Propositions \ref{prop_localization1} and \ref{prop_loc2} are proved 
using the fact that a certain function $w \in H^1(D_1 \cup D_2)^3$
is such that it can be approximated by solutions to the Navier equation 
on $\Omega$. In this subsection we will construct the desired $w$, by using 
a Runge type argument.

Let $D	\Subset \Omega$ with Lipschitz boundary. We start with proving a criterion that tells us 
when a subspace  $\mathcal{S}_D$ of solutions 
of the Navier equation on $D$, can be approximated in $L^2(D)^3$, by solutions in
$$
\mathcal{S}_\Omega := \{ u \in H^1(\Omega)^3 \,:\, u \text{ in solves \eqref{eq_bvp1} with } g \in V^\perp\},
$$
where $V$ is a given finite dimensional subspace of $L^2(\Gamma_N)^3$.

\begin{lem} \label{lem_approxCrit}
Every $\varphi \in \mathcal{S}_D$, where
$$
\mathcal{S}_D := \{ \varphi \in H^1(D)^3 \,:\, L_{\lambda,\mu,\rho} \varphi =  0 \text{ in } D ,\, 
\varphi \perp \mathcal{F}_V \},
$$
can be approximated by elements in $\mathcal{S}_\Omega$ in the $L^2(D)^3$-norm, where $\mathcal{F}_V$ is given by
$$
 \mathcal{F}_V = \spn\big\{ v \in L^2(D)^3 \,:\, 
( S(\chi_D v), \gamma_\mathbb{C} \varphi )_{L^2(\Gamma_N )^3} = 0, \; \forall \varphi \in \mathcal{S}_\Omega \big\},
$$
where $S$ is the solution operator mapping the source term $F$ to the corresponding solution
of the boundary value problem \eqref{eq_bvpLsource}
with $\tau=0$. 
\end{lem}

\begin{proof}
Denote the restrictions of elements of $\mathcal{S}_\Omega$ by
$$
\mathcal{R}_D := \{ \varphi|_D \,:\, \varphi \in \mathcal{S}_\Omega \}.
$$
Now consider  $v \in \mathcal{R}_D^\perp \subset L^2(D)^3$, and let  $\varphi|_D \in \mathcal{R}_D$.
Define $w$ as the weak solution to the source problem
\begin{align}  \label{eq_bvp_w}
\begin{cases}
L_{\lambda,\mu,\rho} w &= \chi_D v , \quad \text{ in } \Omega,\\
\quad \gamma_\C w  |_{\Gamma_N} &= 0, \\	
\quad \quad w  |_{\Gamma_D} &= 0. 
\end{cases}
\end{align}
Now since $v \in \mathcal{R}_D^\perp$, and since both $w$ and $\varphi$ are weak solutions on $\Omega$
(to slightly different equations),
we have that
\begin{align*} 
0 = (v,\varphi|_D)_{L^2(D)^3} 
= - \int_{\Omega} 2 \mu \hat \nabla w :\hat \nabla \varphi 
   + \lambda \nabla \cdot w \nabla \cdot \varphi - \omega^2\rho w \cdot \varphi\,dx 
= - (w,\gamma_\C \varphi)_{L^2(\Gamma_N)^3}.
\end{align*}
We can write  $w = S(\chi_D v)$, where $S$ is the solution operator  implicitly given by Proposition \ref{prop_Neumann_spectrum}
and the fact that we assume that $\omega$ is a non resonance frequency.
Thus we have that 
$$
v \in \mathcal{R}_D^\perp  \quad \Leftrightarrow \quad   
v \in \spn \big\{ v \in L^2(D)^3 \,:\, 
( S(\chi_D v), \gamma_\mathbb{C} \varphi )_{L^2(\Gamma_N )^3} = 0, \; \forall \varphi \in \mathcal{S}_\Omega \big\}.
$$
We thus see that 
\begin{align*} 
u \in \overline{\mathcal{R}_D} = \mathcal{R}_D^{\perp \perp} 
&\;\Leftrightarrow\; u \perp \mathcal{R}_D^\perp  \\
&\;\Leftrightarrow\; u \perp 
 \spn\big\{ v \in L^2(D)^3 \,:\, 
( S(\chi_D v), \gamma_\mathbb{C} \varphi )_{L^2(\Gamma_N )^3} = 0, \; \forall \varphi \in \mathcal{S}_\Omega \big\},
\end{align*}
which proves the claim.

\end{proof}

\noindent
We can now use Lemma \ref{lem_approxCrit} and Corollary \ref{cor_divNonZero} to
construct a solution $w$ to the Navier equation,
that we can Runge approximate with the restricted boundary data, and which has
non-zero divergence in a specified subset.

\begin{lem} \label{lem_blowupCand}
Let $D_1$ and $D_2$ be as in \eqref{eq_Dassump} and 
$\mathcal{S}_D$ as in Lemma \ref{lem_approxCrit}. There exists a solution 
$w \in \mathcal{S}_D$, such that  $w=0$ in $D_1$ and 
$$
\| w \|_{ L^2(D_2)^3 }, 
\;\| \hat \nabla w \|_{ L^2(D_2)^{3\times 3} }, 
\;\| \nabla \cdot w \|_{ L^2(D_2)} \neq 0.
$$
\end{lem}

\begin{proof} 
Let $D = D_1 \cup D_2$. We will show that we can pick the required $w \neq 0$ as
$$
L_{\lambda,\mu,\rho}w_{\tilde g} = 0,\quad \text{ in } D,
$$
where we set $w_{\tilde g} \equiv  0$, in $D_1$, so that
\begin{align}  \label{eq_wBVP0}
L_{\lambda,\mu,\rho} w_{\tilde g} = 0 \quad \text{ in } D_1,\qquad 
\gamma_\mathbb{C} w_{\tilde g} = 0\quad  \text { on } \p D_1,
\end{align}
and in $D_2$ we let $w_{\tilde g}$ solve the boundary value problem
\begin{align}  \label{eq_wBVP}
L_{\lambda,\mu,\rho} w_{\tilde g} = 0 \quad \text{ in } D_2,\qquad 
\gamma_\mathbb{C} w_{\tilde g} = {\tilde g}\quad  \text { on } \p D_2.
\end{align}
We derive four conditions on the boundary value ${\tilde g} \in L^2(\p D_2)^3$,
which will guarantee that
we have the desired properties, and then show that these conditions can be satisfied.

\medskip
\noindent
\textit{Condition 1.} It might happen that zero is a Neumann eigenvalue on $D_2$ we thus require that 
\begin{align}  \label{eq_condN}
{\tilde g} \perp \mathcal{N}_0, 
\end{align}
where $\mathcal{N}_0$ is the finite dimensional subspace of Corollary \ref{cor_ZeroEigenvalue}
and otherwise we set $\mathcal{N}_0 = \{0\}$. This guarantees that we have a unique solution, 
and that we can define the Neumann-to-Dirichlet map on $\p D_2$.

\medskip
\noindent
\textit{Condition 2.} 
The second condition on ${\tilde g}$ will guarantee that $w_{\tilde g} \in \mathcal{S}_D$.
We will derive this next.
In order to obtain a $w_{\tilde g} \in \mathcal{S}_D$, we know by  the definition of $\mathcal{S}_D$ in
Lemma \ref{lem_approxCrit} that it is enough
to pick a boundary condition ${\tilde g}\neq 0$ resulting in a $w_{\tilde g}$ that satisfies the condition
$$
w_{\tilde g} \perp \mathcal{F}_V.
$$
The first step is to rewrite this as a condition on the boundary value ${\tilde g}$. 
Suppose that the condition holds for $w_{\tilde g}$ solving \eqref{eq_wBVP0}--\eqref{eq_wBVP}.
For all $v \in \mathcal{F}_V$ and corresponding $w$ that solve \eqref{eq_bvp_w}, we have that
\begin{align*} 
0 = (w_{\tilde g}, v)_{L^2(D)^3} &= -\int_{D_2} 2 \mu \hat \nabla w_{\tilde g} : \hat \nabla w 
+ \lambda (\nabla \cdot w_{\tilde g}) (\nabla \cdot w) - \omega^2 \rho w_{\tilde g} \cdot w \, dx \\
&\quad + \int_{\p D_2} \gamma_\C w \cdot w_{\tilde g} \,dS
\end{align*}
where we used the fact that $w$ restricted to $D_2$ is a solution in $D_2$ to the boundary value problem 
implied by \eqref{eq_wBVP0}.
Furthermore if we let $\Lambda_{D_2}$ be the Neumann-to-Dirichlet map on $D_2$, then
$$
\int_{\p D_2} \gamma_\C w \cdot w_{\tilde g} \,dS  
=  \big( \gamma_\C w ,\,\Lambda_{D_2} \tilde g \big)_{L^2(\p D_2)^3}
=  \big( \Lambda_{D_2} \gamma_\C w ,\, \tilde g \big)_{L^2(\p D_2)^3}.
$$
Using this and the fact that $w_{\tilde g}$ is a solution in $D_2$ with the boundary condition $\tilde g$,
we get from the above that
$$
0 = \big( \Lambda_{D_2} \gamma_\C w ,\, \tilde g \big)_{L^2(\p D_2)^3} -\big( w ,\, \tilde g \big)_{L^2(\p D_2)^3}
 = \big( \Lambda_{D_2} \gamma_\C S(v) - S(v) ,\, \tilde g \big)_{L^2(\p D_2)^3},
$$
where we write $w = S(v)$, using $S$ which is the solution operator of \eqref{eq_bvp_w}. 
It follows that ${\tilde g}$ should satisfy the constraint 
\begin{align}  \label{eq_condW}
\tilde g \perp \mathcal{W} \,:=\, \spn \{\,(\Lambda_{D_2}\gamma_{\mathbb{C}, \p D_2} - \operatorname{Tr}_{\p D_2}) \circ \,S(v) 
\,:\, v \in \mathcal{F}_V \},
\end{align}
where $\operatorname{Tr}$ denotes the trace operator.
Next we show that this is a finite dimensional constraint on ${\tilde g}$, in the sense that $\dim{\mathcal{W}}< \infty$.
It is enough to show that
\begin{align}  \label{eq_Sv_dim}
N:= \dim \big(\spn \{ S(v)|_{\Omega\setminus D_2} \,:\, v \in \mathcal{F}_V \} \big) < \infty,
\end{align}
since $\mathcal{W}$ can be obtained from this by a linear map.
Note firstly that by definition
\begin{align*} 
v \in \mathcal{F}_V \quad \Rightarrow \quad S(v)|_{\Gamma_N} \perp  V^\perp
\quad\Rightarrow \quad S(v)|_{\Gamma_N}  \in V^{\perp\perp} = V,
\end{align*}
so that $S(\cdot)|_{\Gamma_N}$ here maps into $V$.
Suppose that $N > \dim(V)$. Then we have an $S(v)|_{\Omega\setminus D_2} \neq 0$, 
such that 
$$
S(v)|_{\Gamma_N} = 0
$$
since the trace map is linear. Moreover since $S(v)$ has a zero Neumann condition,
we have that  $\gamma_{\C} S(v)|_{\Gamma_N} = 0$. Thus
$$
S(v)|_{\Gamma_N} = 0, \qquad \gamma_{\C} S(v)|_{\Gamma_N} = 0.
$$
The unique continuation principle of Proposition \ref{prop_bndryUCP} give now that $S(v_0) \equiv 0$,
which is a contradiction.

\medskip
\noindent
\textit{Condition 3:} 
The third condition on ${\tilde g}$ will guarantee that $\nabla \cdot w_{\tilde g} \neq 0$.
For this we will use Corollary \ref{cor_divNonZero}. We require that
\begin{align*} 
{\tilde g} \in \mathcal{T} 
:&= \{ g \in \mathcal{N}_0^\perp \;:\;  \nu \times  \Lambda_{D_2} g = 0 ,\;   \Lambda_{D_2}g   \perp \mathcal{N} \},
\end{align*}
where $\Lambda_{D_2}$ is the Neumann-to-Dirichlet map on $\p D_2$ and $\mathcal{N}$ is as in  Corollary \ref{cor_divNonZero}.
Notice that since $w_{\tilde g}|_{\p D_2} = \Lambda_{D_2} \tilde g$,
this implies by Corollary \ref{cor_divNonZero} that $\nabla \cdot  w_{\tilde g} \neq 0$, when $\tilde g \neq 0$.
We claim that $\mathcal{T}$ is infinite dimensional. To see this consider the space
of Dirichlet conditions
$$
\mathcal{D} := \{ f \in H^{1/2}(\p D_2)^3\;:\;  \nu \times f = 0 ,\;  f \perp \mathcal{N} \}.
$$
Clearly $\dim( \mathcal{D}) = \infty$.
Any (modulo a finite dimensional subspace, in case zero is a Dirichlet eigenvalue)
$f_k \in \mathcal{D}$ corresponds to a $g_k \in \mathcal{T}$, such that
$$
\Lambda_{D_2} g_k = f_k.
$$
Since $\Lambda_{D_2}$ restricted to $\mathcal{T}$ has thus an image that is infinite dimensional,
we have that $\dim( \mathcal{T}) = \infty$.

\medskip
\noindent
\textit{Condition 4.} 
The final requirement on ${\tilde g}$ will be that 
\begin{align}  \label{eq_condR}
{\tilde g} \perp \mathcal{R}, 
\end{align}
where $\mathcal{R}$ is the space of rigid motions.

\medskip
\noindent
The requirements on $\tilde g$ in \eqref{eq_condR}, \eqref{eq_condN} and \eqref{eq_condW}
can be written as 
$$
{\tilde g} \in \mathcal{T}\qquad \text{ and }\qquad 
{\tilde g} \;\perp\;  \mathcal{N}_0 \oplus \mathcal{W} \oplus \mathcal{R}.
$$
These conditions can be met since $\dim (\mathcal{N}_0 \oplus \mathcal{W} \oplus \mathcal{R}) < \infty$ is 
and  $\dim(\mathcal{T})= \infty$.
The final step is to check that $w_{\tilde g}$ satisfies the desired properties. Since ${\tilde g} \perp \mathcal{W}$, 
we have that $w_{\tilde g} \in \mathcal{S}_D$. Because $\tilde g \in \mathcal{T}$ we have that
$\nabla \cdot w_{\tilde g} \neq 0$
and thus $\| \nabla \cdot w_{\tilde g} \|_{ L^2(D)^3 } \neq 0$ and $\|  w_{\tilde g} \|_{ L^2(D)^3 } \neq 0$.
Moreover we have that ${\tilde g} \perp \mathcal{R}$ in
$L^2(D)^3$, 
hence we get by Lemma \ref{lem_symGradient} that
$$
\| \nabla w_{\tilde g} \|_{ L^2(D)^3 } \leq c \| \hat \nabla w_{\tilde g}  \|_{ L^2(\Omega)^3 }.
$$
This implies that $\| \hat \nabla w_{\tilde g} \|_{ L^2(\Omega)^{3 \times 3}  } \neq 0$.

\end{proof}

\section{Shape reconstruction}\label{sec_shape}

\noindent 
In this section we will justify the individual steps in Algorithm \ref{alg_shapeInclusion}. 
We will first quickly in subsection \ref{sec_density_perturbation} consider  
the simpler case of a perturbation in the density and Lamé parameters that are known.
The main arguments are in subsection \ref{sec_mult}, where
we consider the case in which multiple  parameters may have been perturbed.

\medskip
\noindent
We begin however with establishing some preliminaries.
We extend the use of the term \emph{outer support} that was formulated in \cite{HU13} for measurable functions,
to measurable sets.
The outer support (with respect to $\p\Omega$) of a measurable function $f: \Omega \to \R$   is defined as 
$$
\operatorname{osupp} (f) := \Omega \setminus \bigcup \, \big \{  U \subset \Omega \,:\, U  \text{ is relatively open and connected to $\p \Omega$},
f|_U \equiv 0  \big\}.
$$
It will be convenient to extend this definition to sets. We define the outer support of a measurable
set $D \subset \Omega$ 
(with respect to $\p\Omega$) as
\begin{align}  \label{eq_def_osupp}
\osupp (D) := \osupp(\chi_D)
\end{align}
where $\chi_D$ is the characteristic function of the set $D$.

\medskip
\noindent
We will also use the following  elementary lemma for compact self-adjoint operators, which
we will later apply to the differences between Neumann-to-Dirichlet maps.

\begin{lem} \label{lem_finV} 
Suppose $\mathcal{H}$ is a separable Hilbert space and 
$T:\mathcal{H} \to \mathcal{H}$ is a compact self-adjoint operator.
\begin{enumerate}[label=(\alph*)]
\item If $T$ has finitely many negative eigenvalues, then there exists a subspace $W \subset \mathcal{H}$,
such that  $\dim (W)<\infty$, and
$$
(Tg,g)_{\mathcal{H}} \geq 0, \qquad g \in  W^\perp.
$$
\item Assume  $W \subset \mathcal{H}$ is a subspace, with $\dim(W) < \infty$ and that
$$
(Tg,g)_{\mathcal{H}} \geq 0, \qquad g \in  W^\perp.
$$
Then $T$ has finitely many negative eigenvalues $\tau_k$.
Furthermore 
$$
\dim(W) \geq \dim( \spn \{ \varphi_k \,:\, \tau_k < 0 \})
$$
where $\varphi_k$ are the corresponding eigenfunctions.
\end{enumerate}
\end{lem}

\begin{proof}
Note first that according to the spectral Theorem for compact self-adjoint operators,
see \cite{RSI} Theorem VI.16, the eigenfunctions $\varphi_k$ are an orthonormal basis, and 
$$
T g = \sum_k \tau_k c_k \varphi_k,\quad \text{ where } \quad c_k = (\varphi_k,g)_{\mathcal{H}}, 
\quad g \in \mathcal{H},
$$
and where $\tau_k$ are the eigenvalues of $T$.
We now prove $(a)$. Set 
$$
W = \spn \{ \varphi_k \;:\; \tau_k < 0 \} = \{\varphi_1,\dots,\varphi_N\}.
$$
Suppose that $g \in W^\perp$. Then $(\varphi_k,g)_{\mathcal{H}} = 0$ for $k=1,..,N$, so that
$$
g = \sum_{k=N+1}^\infty c_k \varphi_k.
$$
To see that the claim holds for $W$, notice that
\begin{align*} 
(Tg,g)_{\mathcal{H}} = \Big(\sum_{k\geq N+1} \tau_k c_k \varphi_k,\, \sum_{j\geq N+1}   c_j \varphi_j \Big)_{\mathcal{H}}
 = \sum_{k=N+1} \tau_k c_k^2 \geq 0,
\end{align*}
since $\tau_k \geq 0$ for $k \geq N+1$.

\medskip
\noindent
Next we prove $(b)$. 
Let $\mathcal{N} = \operatorname{span}\{\varphi_k \,:\, \tau_k < 0 \}$ 
and assume that $ \dim(\mathcal{N}) = \infty$.
Consider the projection $P_\mathcal{N} W \subset \mathcal {N}$, i.e. the 
projection of the subspace $W$ to $\mathcal{N}$.
This is a subspace and
notice that $\dim (P_\mathcal{N} W) < \dim(W) = M < \infty$.
Since $\dim (\mathcal{N}) = \infty$, we can find a non-zero vector 
$u \in \mathcal{N} \cap (P_\mathcal{N}W)^\perp$\footnote{where
the orthogonal complements and inner products are induced by the $\mathcal{H}$ inner product.}. 
A projection operator is a symmetric operator,
and since $u = P_\mathcal{N}u$ we have for all $w \in W$, that
$$
(u,w)_\mathcal{H}
=(P_\mathcal{N}u,w)_\mathcal{H}
=(u,P_\mathcal{N}w)_\mathcal{H} = 0.
$$
Thus $u \in W^\perp$. We can furthermore write $u$ as 
$$
u = 	\sum_{\varphi_k \in \mathcal{N}} c_k \varphi_k.
$$
We now have that
\begin{align*} 
(Tu,u)_{\mathcal{H}} 
= \Big(\sum_{\varphi_k \in \mathcal{N}} \tau_k c_k \varphi_k,\, \sum_{\varphi_j \in \mathcal{N}}   c_j \varphi_j \Big)_{\mathcal{H}}
 = \sum_{\varphi_k \in \mathcal{N}} \tau_k c_k^2  < 0.
\end{align*}
This is a contradiction since we assumed that if $g \in W^\perp$, then  $(Tg,g)_{\mathcal{H}} \geq 0$.

\medskip
\noindent
As the last step we will derive the lower bound for $\dim(W)$. Assume the bound does not hold.
Then there is a $w \in \mathcal N := \spn \{ \varphi_k \,:\, \tau_k < 0 \}$, such that $w \in W^\perp$.
Using the fact that $\{\varphi_k\}$ is an orthonormal basis, we have that
\begin{align*} 
(Tw,w)_{\mathcal{H}} 
= \Big(\sum_{\varphi_k \in \mathcal{N}} \tau_k c_k \varphi_k,\, \sum_{\varphi_j \in \mathcal{N}}   c_j \varphi_j \Big)_{\mathcal{H}}
 = \sum_{\varphi_k \in \mathcal{N}} \tau_k c_k^2  < 0.
\end{align*}
where $w = \sum_{\varphi_k \in \mathcal{N}} c_k \varphi_k$. This is a contradiction, since 
$w \in W^\perp$ and by the assumption
$$
(Tw,w)_{\mathcal{H}} \geq 0.
$$
\end{proof}

\noindent
In the sequel it will be important that $\Lambda$ is compact and self-adjoint, 
as the following Lemma shows. This will in particular  guarantee that $\Lambda$
has the correct spectral properties.

\begin{lem} \label{lem_ND_cmpct_SA}
The mapping  $\Lambda:L^2(\Gamma_N)^3 \to L^2(\Gamma_N)^3$ is compact and self-adjoint.
\end{lem}

\begin{proof}
We begin by proving that $\Lambda$ is compact. Let $g \in L^2(\Gamma_N)^3$.
The solution $u$ to \eqref{eq_bvp1} depends continuously on the boundary data, 
and by the apriori estimate \eqref{eq_aprioriEst}, we have that 
$$
\| u \|_{H^1(\Omega)^3} \leq C \| g \|_{ L^2(\Gamma_N)^3}.
$$
This together with the continuity of the trace, gives that
$$
\| \Lambda g \|_{ H^{1/2}(\Gamma_N)^3 } 
=
\| u |_{\p \Omega} \|_{ H^{1/2}(\p \Omega)^3 } 
\leq C 
\| u \|_{H^1(\Omega)^3}
\leq C 
\| g \|_{L^2(\Gamma_N)^3}.
$$
Thus $\Lambda : L^2(\Gamma_N)^3 \to H^{1/2}(\Gamma_N)^3$ is continuous. 
It follows that $\Lambda : L^2(\Gamma_N)^3 \to L^{2}(\Gamma_N)^3$
is compact, since  the inclusion mapping $i:H^{1/2}(\Gamma_N)^3 \to L^{2}(\Gamma_N)^3$ is compact. 

Next we will show that $\Lambda$ is a self-adjoint operator. 
It is enough to check that $\Lambda$ is symmetric, since $\Lambda$ is bounded on $L^2(\Gamma_N)^3$.
Let $f,g \in L^2(\Gamma_N)^3$.
Let $u_g \in\mathcal{V}$ be a weak solution to \eqref{eq_bvp1} with boundary data $g$,
and $u_f\in\mathcal{V}$ be a weak solution to \eqref{eq_bvp1} with boundary data $f$.
By the definition of weak solutions \eqref{eq_weak2}, we have that
$$
B(u_f,u_g)  
= -\int_{\Gamma_N} f \cdot u_g \,dS 
= -\big(f\,,\Lambda g\big)_{L^2(\Gamma_N)^3}.
$$
Likewise we have that 
$$
B(u_g,u_f)  
= -\int_{\Gamma_N} g \cdot u_f \,dS 
= -\big(g\,,\Lambda f\big)_{L^2(\Gamma_N)^3}.
$$
By the symmetry of $B$ we have thus that 
$$
\big(\Lambda f\,,g\big)_{L^2(\Gamma_N)^3} = \big(f,\,\Lambda g\big)_{L^2(\Gamma_N)^3},
$$
and we thus see that $\Lambda$ is symmetric.
\end{proof}

\subsection{Recovering inclusions in the density} \label{sec_density_perturbation}

\noindent
In this short subsection we study the shape reconstruction problem in the case when 
only the density is perturbed.
The purpose of this subsection is twofold. Firstly 
we show that we can derive an inclusion detection test for an inhomogeneity
with increasing\footnote{A similar result could be established for decreasing density as in the
next subsection but we do not pursue this here.} density, when the  Lamé parameters are known 
but not necessarily constant.
The argument here also partly serves to illustrate the procedure 
in the next subsection in a simpler context. Here we only derive the analogue of Theorem \ref{thm_inclusionDetection_rho_mu}
and omit proving the analogue of Theorem \ref{thm_inclusionDetection_EVbound}.

\medskip
\noindent
Let $D \Subset \Omega$ with Lipschitz boundary.
We will now assume that\footnote{This assumption
can be relaxed to $\lambda,\mu \in  L^\infty_+(\Omega)$
by the smoothing argument in the next subsection.} 
$\lambda,\mu \in W^{1,\infty}(\Omega) \cap L^\infty_+(\Omega)$
and $\rho_0 > 0$.
Now consider $\rho \in L_+^\infty(\Omega)$ which is such that 
$$
\rho(x) = \rho_0 + \chi_D(x) \psi(x), \quad \psi \in L^\infty(\Omega),\quad \psi(x) > m_0 > 0.
$$
Here $\psi\chi_D$ models an inhomogeneity in an otherwise homogeneous background density. 
The following Theorem gives a criterion for when a set $B\subset \Omega$ 
is a subset of $\osupp(D)$. Note that $B$ might be misscharacterized,
if $B$ is a subset of a component of $\Omega \setminus D$ that is not connected  to
the boundary, so that internal cavities in the inhomogeneous region are not necessarily 
classified correctly.
For a more elaborate explanation on how to use the following theorem 
for inclusion detection see the next section.

\begin{thm} \label{thm_inclusionDetection}
Let $B \subset \Omega$ be an open set and $\alpha > 0$. Let the Neumann-to-Dirichlet maps
$$
\Lambda := \Lambda_{\lambda,\mu,\rho} \qquad \text{ and } \qquad \Lambda^\flat := \Lambda_{\lambda,\mu, \rho_0 + \alpha \chi_B}
$$
where the r.h.s. Neumann-to-Dirichlet maps are those of 
boundary value problem \eqref{eq_bvp1}, with coefficients indicated by the indices.
The following holds: \\
\begin{enumerate}
\item If $B\subset D$, then for all $\alpha\leq m_0$ the map
$\Lambda - \Lambda^\flat$ has finitely many negative eigenvalues. \\

\item If $B\not \subset \osupp(D)$, then  $\Lambda - \Lambda^\flat$ has infinitely 
many negative eigenvalues for all $\alpha >0$.
\end{enumerate}
\end{thm}

\begin{proof}
Notice firstly that $\Lambda - \Lambda^\flat$ is by Lemma \ref{lem_ND_cmpct_SA} a
compact self-adjoint operator.
We begin by proving $(2)$.
We can assume that $B$ is a ball and that $B \cap \osupp(D) = \emptyset$, by choosing a smaller ball inside $B$.
Assume that the claim is false and that there are only finitely  many negative eigenvalues.
By Lemma \ref{lem_finV} there is a finite dimensional subspace $V_1 \subset L^2(\Gamma_N)^3$, 
for which
\begin{align}  \label{eq_pos1}
\big( (\Lambda - \Lambda^\flat)g,\,g \big)_{L^2(\Gamma_N)^3} \geq 0, \qquad \text{ for } g \in V_1^\perp.
\end{align}
Choosing  $\rho_1 = \rho$, $\rho_2 = \rho_0 + \alpha \chi_B$ and
$\lambda_1=\lambda_2=\lambda$ and $\mu_1=\mu_2=\mu$, 
in Lemma \ref{lem_monotonicity_ineq1}, gives another finite dimensional subspace $V_2$
such that when we rearrange the inequality of the lemma, we have that
\begin{align*} 
\big(  (\Lambda - \Lambda^\flat)g, \,g  \big)_{L^2(\Gamma_N)^3} 
&\leq
\int_\Omega  \omega^2( \rho_1  - \rho_2) |u_1|^2 \,dx \\
&\leq
\int_\Omega  \omega^2(\psi \chi_D  - \alpha \chi_B) |u_1|^2 \,dx  \\
&\leq
\omega^2\int_D \psi |u_1|^2 \,dx - \omega^2\int_B \alpha  |u_1|^2 \,dx, 
\end{align*}
where $u_1$ solves \eqref{eq_bvp1} with coefficients $\lambda$, $\mu$ and $\rho$,
with a boundary conditions that satisfies $g \in V_2^\perp$. 
Now by Lemma \ref{prop_localization1} we can choose a sequence  of solutions $u_{1,j}$ such that
$$
\| u_{1,j} \|_{L^2(\osupp(D))^3} \to 0,\qquad \| u_{1,j} \|_{L^2(B)^3} \to \infty,
$$
as $j \to \infty$ and where $g_{1,j} := \gamma_\C u_{1,j}|_{\Gamma_N} \in (V_1 \oplus V_2)^\perp$,
since $\dim(V_1\oplus V_2) < \infty$.
The previous estimate gives us then that
\begin{align*} 
\big(  (\Lambda - \Lambda^\flat)g_j, \,g_j  \big)_{L^2(\Gamma_N)^3} 
\leq - \tfrac{1}{2} \omega^2 \int_B \alpha  |u_{1,j}|^2 \,dx < 0,
\end{align*}
for big enough $j$. This is  in  contradiction with \eqref{eq_pos1}, since 
$$
g_j \in (V_1 \oplus V_2)^\perp \subset V_1^\perp,
$$
part $(2)$ of the claim is thus proven.

\medskip
\noindent
Next we prove part $(1)$. Assume that $\alpha \leq m_0$. 
We choose $\rho_1 =  \rho_0 + \alpha \chi_B$ and $\rho_2 = \rho$ in 
Lemma \ref{lem_monotonicity_ineq1}. According to Lemma  \ref{lem_monotonicity_ineq1}
there exists a finite dimensional subspace $V \subset L^2(\Gamma_N)^3$, such that
for $g \in V^\perp$ we have that
\begin{align*}  
\big(  (\Lambda - \Lambda^\flat)g, \,g  \big )_{L^2(\Gamma_N)^3} 
\geq
\int_\Omega  \omega^2(\rho - \rho_0 - \alpha \chi_B) |u_1|^2 \,dx
\geq
\int_D  \omega^2(\psi - \alpha \chi_B) |u_1|^2 \,dx
\geq 0,
\end{align*}
since $\psi-\alpha\chi_B \geq m_1 -\alpha \geq 0$.
We thus have that
$$
\big(  (\Lambda - \Lambda^\flat)g, \,g  \big )_{L^2(\Gamma_N)^3} 
\geq 0, \qquad g \in V^\perp,
$$
and by Lemma \ref{lem_finV} that  $\Lambda - \Lambda^\flat$ has finitely many negative eigenvalues,
which proves part $(1)$ of the claim.

\end{proof}

\subsection{Recovery of inclusions in multiple coefficients} \label{sec_mult}

Here we  prove Theorems \ref{thm_inclusionDetection_rho_mu} and \ref{thm_inclusionDetection_EVbound},
and discuss how they can be used to reconstruct the outer support of an inhomogeneities in the
material parameters. This will justify Algorithm \ref{alg_shapeInclusion}.

We will consider inhomogeneities in the material parameters of the following type.
Let $D_1, D_2, D_3 \Subset \Omega$. 
We will now assume that $\lambda,\mu, \rho \in L_+^\infty(\Omega)$ 
are such that
\begin{equation} \label{eq_lambdaMuRho}
\begin{aligned} 
\lambda(x) &= \lambda_0 + \chi_{D_1}(x) \psi_\lambda (x), \qquad \psi_\lambda \in L^\infty(\Omega),
\quad \psi_\lambda(x) > m_1, \\
\mu(x) &= \mu_0 + \chi_{D_2}(x) \psi_\mu(x), \qquad \psi_\mu \in L^\infty(\Omega),\quad \psi_\mu(x) > m_2, \\
\rho(x) &= \rho_0 - \chi_{D_3}(x) \psi_\rho(x), \qquad \psi_\rho \in L^\infty(\Omega),\quad m_3 < \psi_\rho(x) < M_3, 
\end{aligned}
\end{equation}
where the constants $\lambda_0,\mu_0,\rho_0 > 0$ and the bounds  $m_1,m_2, m_3 > 0$ and $\rho_0 > M_3$.
The coefficients $\lambda,\mu$ and $\rho$  model inhomogeneities in an otherwise homogeneous background medium
given by the coefficients $\lambda_0,\mu_0$ and $\rho_0$. 

Next we define the test coefficients $\lambda^\flat,\mu^\flat$ and $\rho^\flat$. Let $B \Subset \Omega$ be a ball. We let
\begin{equation} \label{eq_testCoeff}
\begin{aligned} 
\lambda^\flat(x) &= \lambda_0 + \alpha_1 \chi_{B}(x), \\
\mu^\flat(x) &= \mu_0 + \alpha_2\chi_{B}(x), \\
\rho^\flat(x) &= \rho_0 - \alpha_3\chi_{B}(x), \\
\end{aligned}
\end{equation}
where $\alpha_j \geq  0$ are constants and $\chi_{D_j}$ are characteristic functions, for $j=1,2,3$ .
The following proposition gives a method for recovering $\osupp(D_1\cup D_2\cup D_3)$ from
the Neumann-to-Dirichlet map, and thus the shape of the region where the coefficients differ 
from the background coefficients $\lambda_0,\mu_0$ and $\rho_0$. 

\begin{thm} \label{thm_inclusionDetection_rho_mu}
Let $D := D_1 \cup D_2 \cup D_3$ where the sets are as in \eqref{eq_lambdaMuRho} and
$B \subset \Omega$  and $\alpha_j > 0$ be as in \eqref{eq_testCoeff}, 
and set $\alpha:=(\alpha_1,\alpha_2,\alpha_3)$.
The following holds: \\
\begin{enumerate}

\item Assume that  $B \subset D_j$, for $j \in I$, for some $I \subset\{1,2,3\}$.  
Then for all $\alpha_j$ with $\alpha_j \leq m_j$, $j \in I$, and $\alpha_j = 0$, 
$j \notin I$, the map $\Lambda^\flat - \Lambda$ has finitely many negative eigenvalues. \\

\item If $B \not \subset \osupp(D)$, then for all $\alpha$, $|\alpha| \neq 0$, 
the map  $\Lambda^\flat - \Lambda$ has infinitely  many negative eigenvalues.

\end{enumerate}
where $\Lambda$ is the Neumann-to-Dirichlet map the coefficients in \eqref{eq_lambdaMuRho}
and $\Lambda^\flat$ is the Neumann-to-Dirichlet map for the coefficients in \eqref{eq_testCoeff}.
\end{thm}

\begin{proof}
Notice firstly that $\Lambda - \Lambda^\flat$ and $\Lambda - \Lambda^\sharp$  are  by Lemma \ref{lem_ND_cmpct_SA} 
compact self-adjoint operators.
We begin by proving $(2)$.
We can assume that $B \cap \osupp{D} = \emptyset$, by making the ball $B$ smaller if needed.
Assume that the claim is false and thus that there are only finitely 
many negative eigenvalues.

We will use the localized solutions of Proposition \ref{prop_loc2}. The test coefficients $\lambda^\flat,\mu^\flat$ and
$\rho^\flat$ are however not smooth enough. To deal with this we choose a ball $B_0 \Subset B$ and a
$\phi \in C^\infty(\Omega)$.
$$
\phi(x) = 1, \; x \in B_0, \quad \phi(x) = 0, \; x \in \Omega \setminus B, \quad 0 \leq \phi(x) \leq 1,\; x \in \Omega.
$$
Now define
\begin{equation} \label{eq_testCoeff_smooth}
\begin{aligned} 
\lambda^\sharp(x) &= \lambda_0 + \alpha_1 \phi(x), \\
\mu^\sharp(x) &= \mu_0 + \alpha_2\phi(x), \\
\rho^\sharp(x) &= \rho_0 - \alpha_3\phi(x), \\
\end{aligned}
\end{equation}
where $\alpha_j$ are as in \eqref{eq_testCoeff}. Clearly we have that
$$
\lambda^\flat  \geq \lambda^\sharp,\quad
\mu^\flat \geq \mu^\sharp, \quad
\rho^\sharp \geq \rho^\flat.
$$
Thus by Lemma \ref{lem_LambdaMono} there is a subspace $V_1$ for which
\begin{align}  \label{eq_LambdaMono2}
((\Lambda^\sharp-\Lambda) g,g)_{L^2(\Gamma_N)^3} \geq 
((\Lambda^\flat -\Lambda) g,g)_{L^2(\Gamma_N)^3}, \qquad \forall g \in V_1^\perp,
\end{align}
where $\Lambda^\flat$ is the Neumann-to-Dirichlet map corresponding to the coefficients in \eqref{eq_testCoeff}
and where $\Lambda^\sharp$ corresponds to the coefficients in \eqref{eq_testCoeff_smooth}.

We assumed that $\Lambda^\flat - \Lambda$ has finitely many negative eigenvalues.
According  to Lemma \ref{lem_finV}, we then have that there is a finite 
dimensional subspace $V_2 \subset L^2(\Gamma_N)^3$, such that
\begin{align}  \label{eq_pos}
\big( (\Lambda^\flat - \Lambda)g,\,g \big)_{L^2(\Gamma_N)^3} \geq 0 \qquad \text{ for } g \in V_2^\perp.
\end{align}
To obtain a contradiction we consider Lemma \ref{lem_monotonicity_ineq1}, 
where $\Lambda_1 = \Lambda^\sharp$ and $\Lambda_2 = \Lambda$ and which is rearranged to give
\begin{align*} 
\big(  (\Lambda^\sharp - \Lambda)g, \,g  \big)_{L^2(\Gamma_N)^3} 
&\leq
\int_\Omega 
2(\mu-\mu^\sharp ) |\hat \nabla u_1 |^2 + (\lambda - \lambda^\sharp ) |\nabla \cdot u_1|^2  
+ \omega^2(\rho^\sharp - \rho ) |u_1|^2 \,dx  \\
&\leq
\int_\Omega  
2(\psi_\mu\chi_{D_2} - \alpha_2 \phi )|\hat \nabla u_1|^2 
+(\psi_\lambda\chi_{D_1} - \alpha_1 \phi )|\nabla \cdot u_1|^2   \\
&\qquad+ \omega^2(-\alpha_3 \phi + \psi_\rho\chi_{D_3}  ) |u_1|^2 \,dx,
\end{align*}
where $u_1$ solves \eqref{eq_bvp1} with coefficients given by \eqref{eq_testCoeff_smooth} 
boundary condition $g \in V_3^\perp$, where $V_3$ is the subspace given by Lemma \ref{lem_monotonicity_ineq1}.

By Proposition \ref{prop_loc2}
we can choose a sequence $g_j  = (\gamma_\C u_{1,j})|_{\p\Omega} \in (V_1 \oplus V_2 \oplus V_3)^\perp$ of
boundary data that give the  solutions $u_{1,j}$ to \eqref{eq_bvp1}, with the coefficients
\eqref{eq_testCoeff_smooth}, such that
$$
\| u_{1,j} \|_{ L^2(D^+)^3 }, 
\;\| \hat \nabla u_{1,j} \|_{ L^2(\osupp(D))^{3\times 3} }, 
\;\| \nabla \cdot u_{1,j} \|_{ L^2(\osupp(D))} 
\to 0, 
$$
where $\osupp(D) \Subset D^+$, and  $D^+ \cap B = \emptyset$,  and such that
$$
\| u_{1,j} \|_{ L^2(B)^3 }, 
\;\| \hat \nabla u_{1,j} \|_{ L^2(B_0)^{3\times 3} }, 
\;\| \nabla \cdot u_{1,j} \|_{ L^2(B_0)} 
\to \infty, 
$$
as $j \to \infty$.
Inserting these solutions to the previous inequality and using that $B \cap \osupp(D) = \emptyset$,
we get that
\begin{align*} 
\big(  (\Lambda^\sharp  - \Lambda)g_j, \,g_j  \big)_{L^2(\Gamma_N)^3} 
&\leq
C \int_{D} |\hat \nabla u_{1,j}|^2 
+|\nabla \cdot u_{1,j}|^2  
+|u_{1,j}|^2 \,dx  \\
&\quad- \int_{B_0} \alpha_1 |\nabla \cdot u_{1,j}|^2 
+ \alpha_2 |\hat \nabla u_{1,j}|^2 
+ \alpha_3 |u_{1,j}|^2 \,dx.
\end{align*}
Since $|\alpha| \neq 0$ and $\alpha_j \geq 0$, we see that
the last integral becomes large and increasingly negative while the first integral vanishes as $j$ grows, and thus
\begin{align*} 
\big(  (\Lambda^\sharp - \Lambda)g_j, \,g_j  \big)_{L^2(\Gamma_N)^3}  < 0,
\end{align*}
for large enough $j$. This is in contradiction with \eqref{eq_pos} and \eqref{eq_LambdaMono2}, 
since  $g_j \in (V_1\oplus V_2 \oplus V_3)^\perp \subset (V_1 \oplus V_2 )^\perp$.
Part $(2)$ of the claim thus holds.

\medskip
\noindent
Next we prove part $(1)$. Assume that $B \subset D_j$ for  $j \in I$. Choose
$0 \leq \alpha_j \leq m_j$ for  $j \in I$ and $\alpha_j = 0$ for $j \notin I$. 
Moreover choose $\Lambda_1 = \Lambda$ and $\Lambda_2 = \Lambda^\flat$ in Lemma \ref{lem_monotonicity_ineq1}. 
According to Lemma  \ref{lem_monotonicity_ineq1}
there exists a finite dimensional subspace $V \subset L^2(\Gamma_N)^3$, such that 
if $g \in V^\perp$, then
\begin{align*}  
\big(  (\Lambda^\flat - \Lambda)g, \,g  \big )_{L^2(\Gamma_N)^3}
&\geq
\int_\Omega 
2(\mu -\mu^\flat ) |\hat \nabla u_1|^2 
+ (\lambda - \lambda^\flat ) |\nabla \cdot u_1|^2  
+ \omega^2(\rho^\flat -\rho) |u_1|^2 \,dx  \\
&\geq
\int_{D_2}  2( m_2 - \alpha_2 \chi_{B} ) |\hat \nabla u_1|^2 \,dx 
+ \int_{D_1} ( m_1 - \alpha_1 \chi_{B} )|\nabla \cdot u_1|^2  \,dx\\
&\quad + \int_{D_3} \omega^2(-\alpha_3\chi_{B} + m_3) |u_1|^2 \,dx  \\
&\geq 0,
\end{align*}
where we use the properties in \eqref{eq_lambdaMuRho}.
We have thus shown that for $\alpha_j$, chosen as 
$0 \leq \alpha_j \leq m_j$ for  $j \in I$ and $\alpha_j = 0$ for $j \notin I$, 
we have that
$$
\big(  (\Lambda^\flat - \Lambda)g, \,g  \big )_{L^2(\Gamma_N)^3} 
\geq 0, \qquad \forall g \in V^\perp.
$$
By Lemma \ref{lem_finV} we have that  $\Lambda^\flat - \Lambda$ has finitely many negative eigenvalues,
which proves part $(1)$.

\end{proof}

\noindent
We can use the previous theorem to give a reconstruction procedure for the outer support of
the inhomogeneous region $D$ as follows.
Choose an open cover $\mathcal{B} = \{ B \}$ of $\Omega$.
By computing $N_B$ which stands for the lowest number of negative eigenvalues of the $\Lambda^\flat - \Lambda$ 
corresponding to a $B$, when varying the parameters as suggested by Theorem 
\ref{thm_inclusionDetection_rho_mu}, we can classify the sets in $\mathcal{B}$. We let 
$$
\mathcal{A} := \{ B \in \mathcal{B} \,:\, N_B < \infty \}.
$$
By Theorem \ref{thm_inclusionDetection_rho_mu} part (2) we have that
$$
B \in \mathcal{A}  \quad\Rightarrow\quad B \subset \osupp(D),
$$
and by Theorem \ref{thm_inclusionDetection_rho_mu} part (1) we have that
$$
B \not\in \mathcal{A}  \quad\Rightarrow\quad B \not \subset D_j,\; j=1,2,3.
$$
So all $B \in \mathcal{A}$ are such that  $B  \subset \osupp(D)$, and for all $B \in \mathcal{B}$
with
$$
B \subset D_j, \quad \text{ for some }j,
$$
we have that $B \in \mathcal{A}$.
Supposing we have chosen $\mathcal{B}$ with enough precision, so that it contains an exact sub cover of $D$,
then 
$$
D = \bigcup_{j=1}^3 D_j \subset \cup \mathcal{A}. 
$$
On the other hand 
$$
\cup\mathcal{A} \subset \osupp(D).
$$
The set $\cup\mathcal{A}$ gives thus  an approximation that contains the inhomogeneous region $D$
and is contained in $\osupp(D)$.
This approximation may differ from $\osupp(D)$, in that sets $B$ that are contained 
in components of $\Omega \setminus D$ that are not connected to $\p \Omega$ (i.e. $B$ that are
in internal cavities in the inhomogeneity)  can be classified incorrectly. 
Note however that, since these possibly misclassified regions are contained 
in $\osupp(D)$ and are disconnected  from $\p \Omega$, we can as a last step
topologically inspect $\mathcal{A}$ and fill in any internal holes, to obtain an approximation
of $\osupp(D)$, that becomes arbitrarily good as we use more precise covers $\mathcal{B}$
of $\Omega$.

\medskip
\noindent
The criterion offered by Theorem \ref{thm_inclusionDetection_rho_mu} and Theorem \ref{thm_inclusionDetection}
suffer from the problem that part (2) of the claims involve an infinite bound. 
We can however replace the bound in Theorem \ref{thm_inclusionDetection_rho_mu} (and Theorem \ref{thm_inclusionDetection}), 
with a finite bound, which we can compute from background coefficients.
More specifically we can derive the following  version
of Theorem \ref{thm_inclusionDetection_rho_mu}.
Note that the bound in the next Theorem is in all likelihood not optimal. 
For more on this see Theorem 1 in \cite{HPS19a}.

\begin{thm}\label{thm_inclusionDetection_EVbound}
Suppose the assumptions of Theorem \ref{thm_inclusionDetection_rho_mu} hold.
Let $M_0\in \R$ be defined as
$$
M_0 := d(\lambda_0,\mu_0,\rho_{0}),
$$
where $d(\lambda_0,\mu_0,\rho_{0})$ is the number of positive eigenvalues of $L_{\lambda_0,\mu_0,\rho_{0}}$ 
as defined in \eqref{eq_def_d}. Then
\begin{enumerate}

\item Assume that  $B \subset D_j$, for $j \in I$, for some $I \subset\{1,2,3\}$.  
Then for all $\alpha_j$ with $\alpha_j \leq m_j$, $j \in I$, and $\alpha_j = 0$, 
$j \notin I$, the map $\Lambda^\flat - \Lambda$ has at most $M_0$ negative eigenvalues. \\

\item If $B \not \subset \osupp(D)$, then for all $\alpha$, $|\alpha| \neq 0$, 
the map  $\Lambda^\flat - \Lambda$ has more than $M_0$ negative eigenvalues,

\end{enumerate}
where $\Lambda$ is the Neumann-to-Dirichlet map the coefficients in \eqref{eq_lambdaMuRho}
and $\Lambda^\flat$ is the Neumann-to-Dirichlet map for \eqref{eq_testCoeff}, and where
the eigenvalues of $\Lambda^\flat - \Lambda$ are counted with multiplicity.
\end{thm}

\begin{proof}
Part (2) of the claim can be proven as part (2) of Theorem \ref{thm_inclusionDetection_rho_mu}.
It remains to prove (1).
As in the proof of (1) in Theorem \ref{thm_inclusionDetection_rho_mu} we have that
$$
\big(  (\Lambda^\flat - \Lambda)g, \,g  \big )_{L^2(\Gamma_N)^3} 
\geq 0, \qquad g \in V^\perp.
$$
By Lemma \ref{lem_finV} we have that  $\Lambda^\flat - \Lambda$ has less or equally as many negative eigenvalues
as $\dim(V)$, when counted with multiplicity. Combining this with the bound on $\dim(V)$ in Lemma \ref{lem_monotonicity_ineq1}, we get that
$$
 \text{number of negative eigenvalues of $\Lambda^\flat - \Lambda$ } \leq \dim(V ) \leq d(\lambda^\flat,\mu^\flat,\rho^\flat). 
$$
We can use the min-max principle to give an upper bound on $\dim(V)$ using $d(\lambda_0,\mu_0,\rho_{0})$
which we can easily compute, and which will prove the claim. We do this by showing that
\begin{align}  \label{eq_dd}
d(\lambda^\flat,\mu^\flat,\rho^\flat) \leq d(\lambda_0,\mu_0,\rho_{0}).
\end{align}
Let $-B^\flat_{\tau_0}$ and $- B^0_{\tau_0}$ be the bilinear forms related to the coefficients
$\lambda^\flat,\mu^\flat,\rho^\flat$ and  $\lambda_0,\mu_0,\rho_{0}$ respectively.
Let $\tau_0 \leq 0$ be such that both $-B^\flat_{\tau_0}$ and $- B^0_{\tau_0}$ are positive definite
bilinear forms. Note that Proposition \ref{prop_Wellposedness} guarantees that such a $\tau_0 \leq 0$ exists. 
Then by the min-max principle of Theorem 3.2 on p. 97 in \cite{Ba80}, we have that
\begin{align*} 
	-\sigma^\flat_k - \tau_0 = \min_{S_k \subset H^1(\Omega)^3} \max_{v \in S_k} \frac{-B^\flat_{\tau_0}(v,v)}{(v,v)_{L^2}}, \qquad	
	-\sigma^0_k - \tau_0 = \min_{S_k \subset H^1(\Omega)^3} \max_{v \in S_k} \frac{- B^0_{\tau_0}(v,v)}{(v,v)_{L^2}}, 	
\end{align*}
where $S_k \subset H^1(\Omega)^3$ is an arbitrary $k$ dimensional subspace, and 
where $\sigma^\flat_k$ are the eigenvalues related to $L_{\lambda^\flat,\mu^\flat,\rho^\flat}$
and $\sigma^0_k$ are the eigenvalues related to $L_{\lambda_0,\mu_0,\rho_{0}}$.
From definition \eqref{eq_testCoeff} we easily see that
$$
- B^0_{\tau_0}(v,v)
\leq
-B^\flat_{\tau_0}(v,v).
$$
Thus $ -\sigma^0_k - \tau_0 \leq -\sigma^\flat_k - \tau_0$, and  hence $\sigma^0_k \geq \sigma^\flat_k$.
It follows  that $L_{\lambda_0,\mu_0,\rho_{0}}$ has more or equally as many positive eigenvalues 
as $L_{\lambda^\flat,\mu^\flat,\rho^\flat}$, which shows that \eqref{eq_dd} holds.

\end{proof}

\begin{rem} \label{rem_apriori}
We need more information than in Theorem \ref{thm_inclusionDetection_EVbound} to derive
a similar bound $M_0$ for Theorem \ref{thm_inclusionDetection}. This is because  in Theorem \ref{thm_inclusionDetection}
the perturbation in the density has a larger value than the homogeneous background density. 
In this case we need to have an apriori bound  $\rho \leq \rho_M$, and use $\rho_M$
instead of $\rho_0$ when computing $M_0$.
\end{rem}

\begin{rem} \label{rem_opposite_inclusions}
A natural extension of  Theorems \ref{thm_inclusionDetection_rho_mu}
and \ref{thm_inclusionDetection_EVbound} is the case where there is an increase in the density $\rho$
and a decrease in the Lamé parameters $\lambda$ and $\mu$ (as in \cite{HPS19b} Theorems 6.1 and 6.2).
We do not however pursue this here as the arguments are similar.
\end{rem}

\begin{rem} \label{rem_non_const_background}
If not all of the parameters $\lambda,\mu$ and $\rho$ are perturbed,
and we assume that we know the unperturbed coefficients, then we could derive
versions of  Theorems \ref{thm_inclusionDetection_rho_mu}
and \ref{thm_inclusionDetection_EVbound}, where the unperturbed coefficients are non-constant as in Theorem \ref{thm_inclusionDetection}. 
\end{rem}

\section{A numerical realization} \label{sec_numerics}

\noindent
In this section, we give the background for the implementation of the
monotonicity tests and present initial numerical results. In more detail, we
introduce a test model with two inclusions and consider different frequencies
for the reconstruction of these inclusions.
\\
\\
We start with the background of the discretization. Let $G=\lbrace g_1,\ldots, g_m\rbrace$ be an orthogonal system of boundary loads. We discretize the Neumann-to-Dirichlet operator by
\begin{align*}
\overline{\Lambda}=\left(\left(\Lambda g_i, g_j\right)\right)_{i,j=1,\ldots,m}.
\end{align*}
\noindent
\\
\noindent
We base the discretization on Theorem \ref{thm_inclusionDetection_EVbound}.
However, we want to remark, that in our discritization we pick an ad hoc chosen
$\tilde M$ that is smaller than the $M_0$ given by Theorem
\ref{thm_inclusionDetection_EVbound} and is chosen by inspecting the number of
negative eigenvalues. The $M_0$ is too large for the accuracy of our current
tests.  We hope to remedy the situation in future work by linearizing and
regularizing the method.
\\
\\
{\bf Discritization:}
\\
Let $D:=D_1 \cup D_2 \cup D_3$, where the sets are as in (\ref{eq_lambdaMuRho})
and $B\subset \Omega$. We chose a nonnegative constant $\tilde{M} \leq  M_0$ and mark
\begin{itemize}

\item[(1)] $B\subset osupp(D)$ if $\overline{\Lambda^\flat}-\overline{\Lambda}$ has at most $\tilde{M}$ many negative eigenvalues
for some $\alpha$, $|\alpha|\neq 0$,

\item[(2)] $B\not\subset D$ if $\overline{\Lambda^\flat}-\overline{\Lambda}$ has more than $\tilde{M}$ negative eigenvalues 
for all $\alpha=(\alpha_1,\alpha_2,\alpha_3)$, $0\leq \alpha_j\leq m_j$, $j=1,2,3$ with $|\alpha|\neq 0$,

\end{itemize}
\noindent
where $\overline{\Lambda}$ is the discrete Neumann-to-Dirichlet operator of the
coefficients in (\ref{eq_lambdaMuRho}) and $\overline{\Lambda^\flat}$ is the
Neumann-to-Dirichlet map for (\ref{eq_testCoeff}).
\\
\\
Based on this, we present some numerical tests and examine an artificial test object with two inclusions (blue) shown in Figure \ref{testobject}. The size of our test object is $1\,m^3$.
\begin{figure}[H]
\centering 
\includegraphics[width=0.32\textwidth]{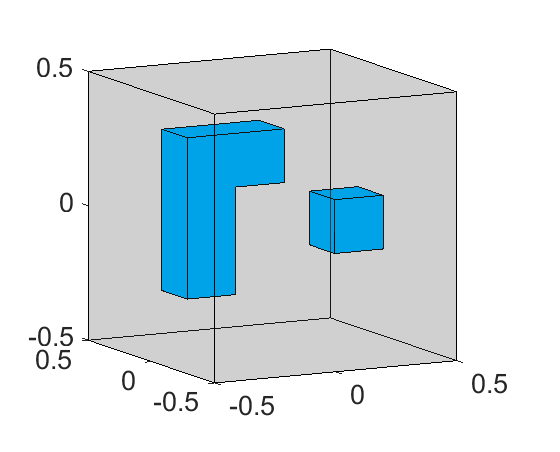}
\caption{Cube with two inclusions (blue) (similar as the model in \cite{EH21}).}\label{testobject}
\end{figure}
\noindent
The parameters of the corresponding materials are given in Table \ref{lame_parameter_mono}
\renewcommand{\arraystretch}{1.3} 
\begin{table} [H]
 \begin{center}
 \begin{tabular}{ |c|c| c |c |}  
\hline
 material & $\lambda_i$ & $\mu_i$ & $\rho_i$\\
  \hline
$i=0$: background &  $6\cdot 10^5$   &  $6\cdot 10^3$  & $3\cdot 10^3$ \\
 \hline
$i=1$: inclusion &  $2\cdot 10^6$ &  $2\cdot 10^4$  & $10^3$ \\
\hline
\end{tabular}
\end{center}
\caption{Lam\'e parameter $\lambda$ and $\mu$ in [$Pa$] and density $\rho$ in [$kg/m^3$].}
\label{lame_parameter_mono}
\end{table}
\noindent
Given an angular frequency $\omega$, the $s$-wavelength and $p$-wavelength (see, e.g., \cite{H01})  for the homogeneous background material are defined via
\begin{align*}
l_p= 2\pi \dfrac{v_p}{k} \quad\text{and}\quad l_s= 2\pi \dfrac{v_s}{k}
\end{align*}
\noindent
with the velocities
\begin{align*}
v_p=\sqrt{\dfrac{\lambda_0+2\mu_0}{\rho_0}} \quad\text{and}\quad v_s=\sqrt{\dfrac{\mu_0}{\rho_0}}.
\end{align*}
\noindent
\\
In order to perform the reconstruction algorithm (see Algorithm \ref{alg_shapeInclusion}), we have to define the required input.
We consider $125$ ($5\times5\times 5$) test inclusions, i.e. cubes. We want to remark that these
test inclusions are chosen in such a way that they perfectly fit into the unknown inclusions.
\\
\\
Further on, we choose for the upper bounds 
\begin{align*}
&\alpha_1=\lambda_1 -\lambda_0=1.4\cdot 10^6 \, Pa, \\
&\alpha_2=\mu_1 -\mu_0=1.4\cdot 10^4 \, Pa, \\
&\alpha_3=\rho_0 -\rho_1=2\cdot 10^3 \, kg/m^3.
\end{align*}
\noindent
For the boundary loads $g_j$, we divide each surface of the cube except its bottom, which
denotes the Dirichlet boundary into $25$ $(5 \times 5)$ squares of equal size. On each of those $125$
squares, we consecutively apply a boundary load of $100\, N/m^2$ in the normal direction
resulting in $125$ boundary loads $g_j$.
\noindent
We want to remark that the simulations are performed with COMSOLMultiphysics with LiveLink for Matlab.
\\
\\
In the following we take a look at different values for $\omega$ and start with
the cases $\omega=10\,rad/s,50\,rad/s$ and the aforementioned $125$ boundary loads.
Figure \ref{plot_eig_k_10} - Figure \ref{plot_eig_k_50} show the number of the
negative eigenvalues of the corresponding test inclusions.

%
%
%

\subsection{Case: $\omega=10$}
\noindent
\\
Figure \ref{plot_eig_k_10} depicts the number of eigenvalues for each test inclusion, where the values of the $x$-axis indicate the index of the $125$ test inclusions and the $y$-axis gives to the number of negative eigenvalues.
\begin{figure}[H]
\centering 
\includegraphics[width=0.37\textwidth]{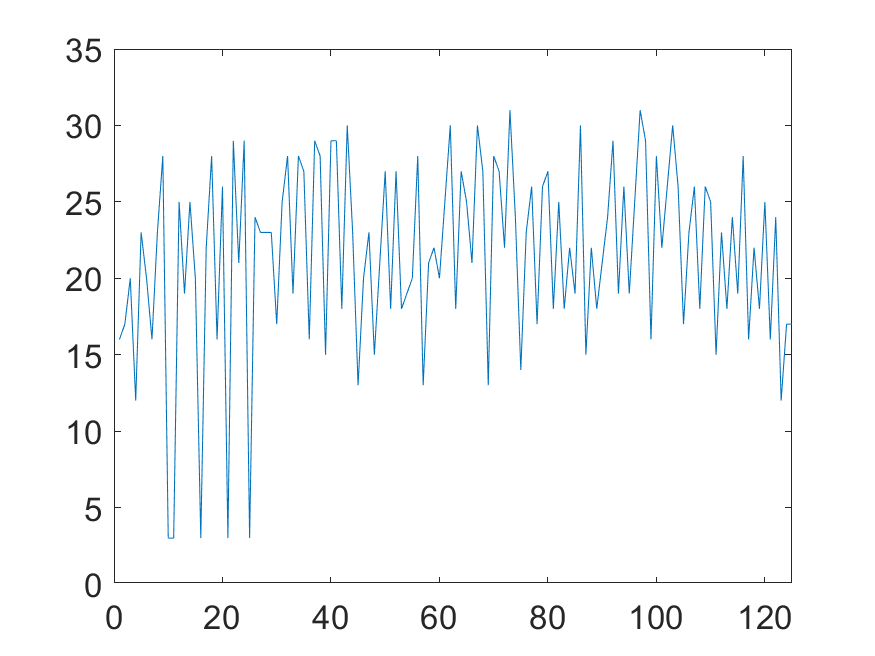}
\caption{Plot of the number of negative eigenvalues for $\omega=10\,rad/s$ ($l_p=9\,m$, $l_s=0.9\,m$ for the homogenous background material).}\label{plot_eig_k_10}
\end{figure}
\noindent
\begin{figure}[H]
\centering 
\includegraphics[width=0.45\textwidth]{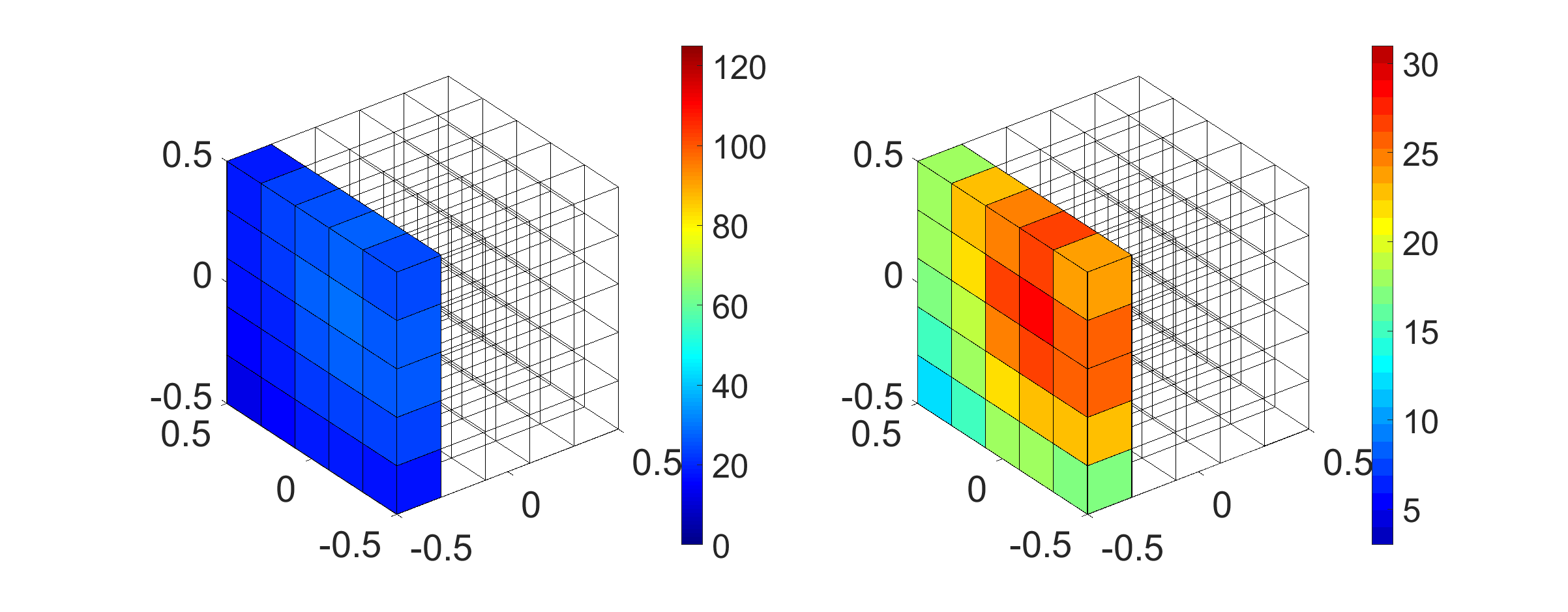}
\includegraphics[width=0.45\textwidth]{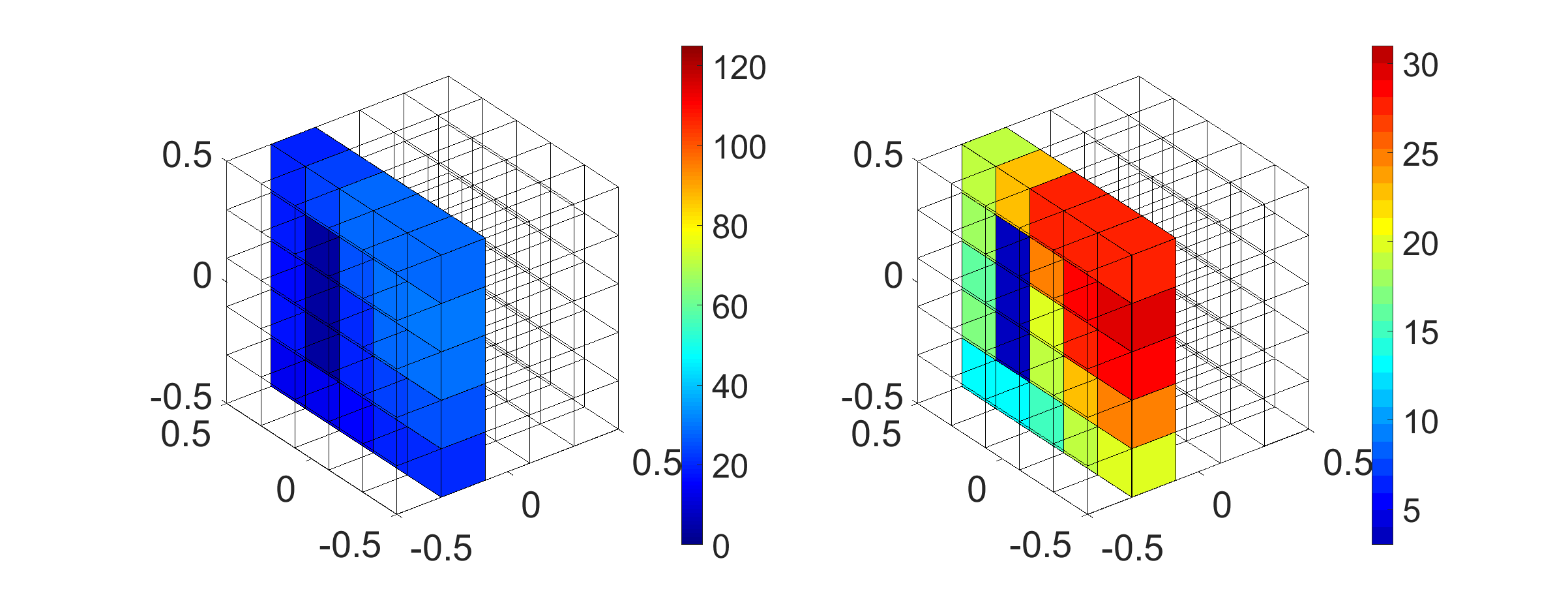}\\
\includegraphics[width=0.45\textwidth]{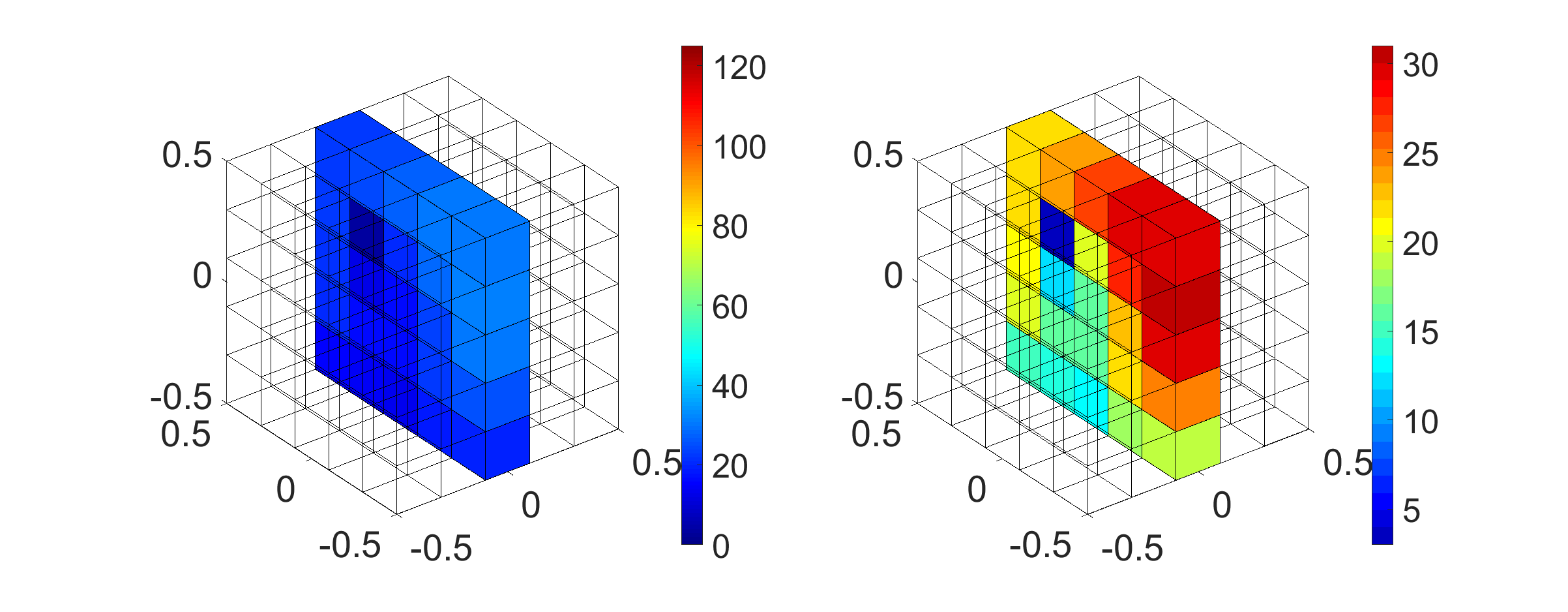}
\includegraphics[width=0.45\textwidth]{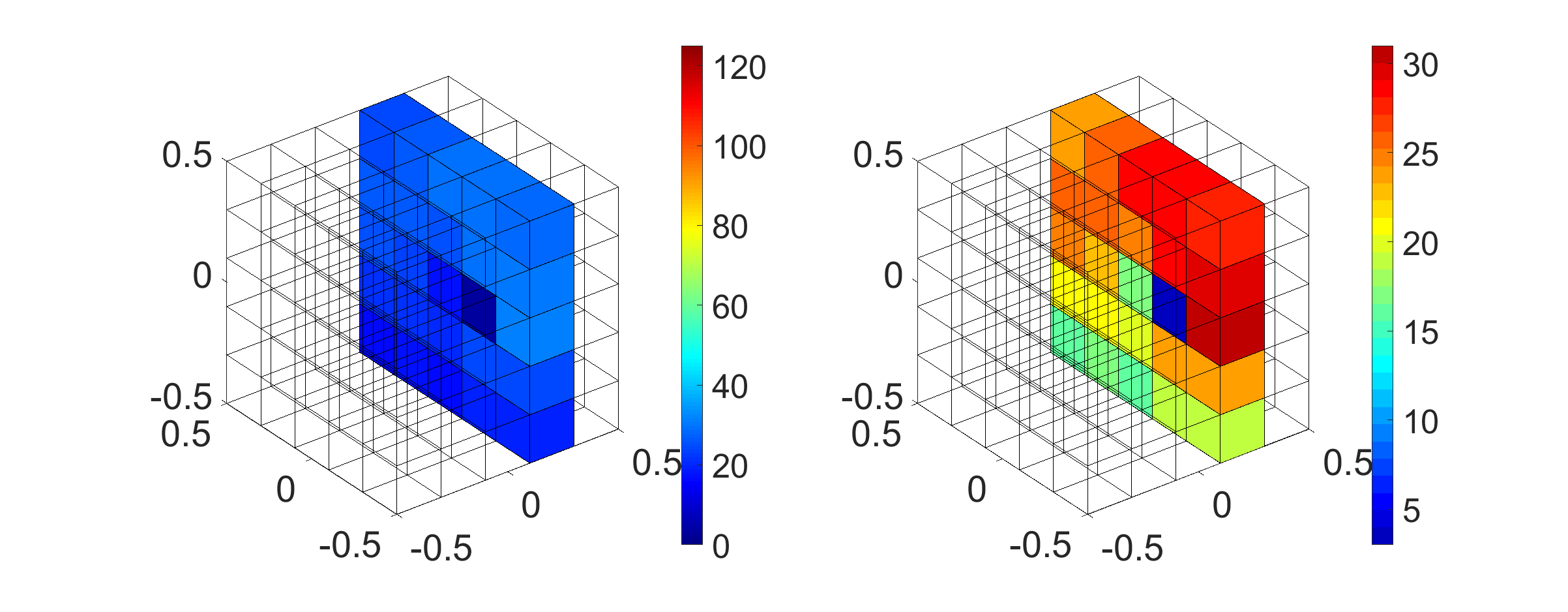}\\
\includegraphics[width=0.45\textwidth]{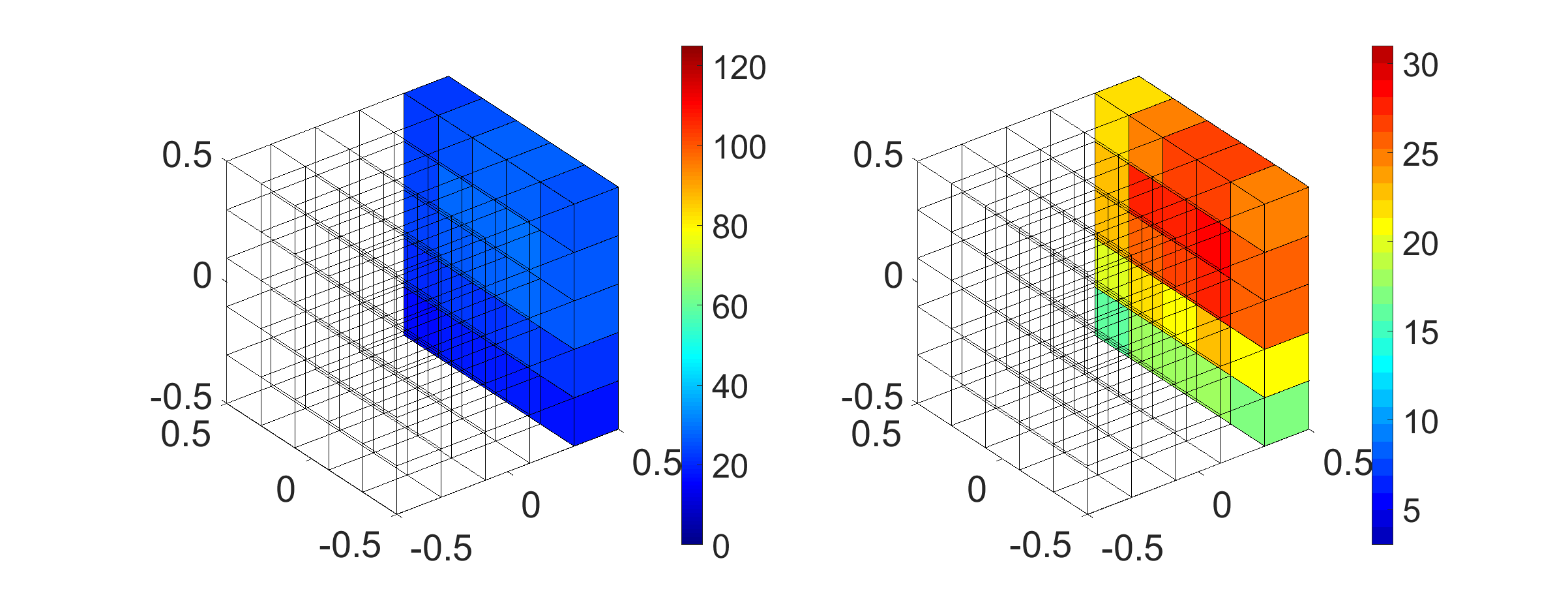}
\caption{Plot of the number of negative eigenvalues for $\omega=10\,rad/s$ w.r.t. each test inclusion: Left hand side: colorbar with complete range, right hand side: colorbar with adopted range}\label{eig_k_10}
\end{figure}
\noindent
We observe that there are distinct spikes in the eigenvalue plots (see Figure
\ref{plot_eig_k_10}) which corresponds to test blocks inside the unknown
inclusions. Hence, the eigenvalue plots suggest a choice of $\tilde{M}$ in the
range between $4$ to $12$ which is less than $M_0=24$ calculated for
$\omega=10\,rad/s$. For more details, the test blocks and their eigenvalues are
depicted in Figure \ref{eig_k_10}, where the blocks which lie inside the
unknown inclusions are clearly visible.

\subsection{Case: $\omega=50$}
\noindent
\\
Similar as before, Figure \ref{plot_eig_k_50} gives us the required information concerning the choice of $\tilde{M}$. Thus, we set $\tilde{M}=50$, where the theoretical $M_0$ was given by $M_0=90$.
\begin{figure}[H]
\centering 
\includegraphics[width=0.37\textwidth]{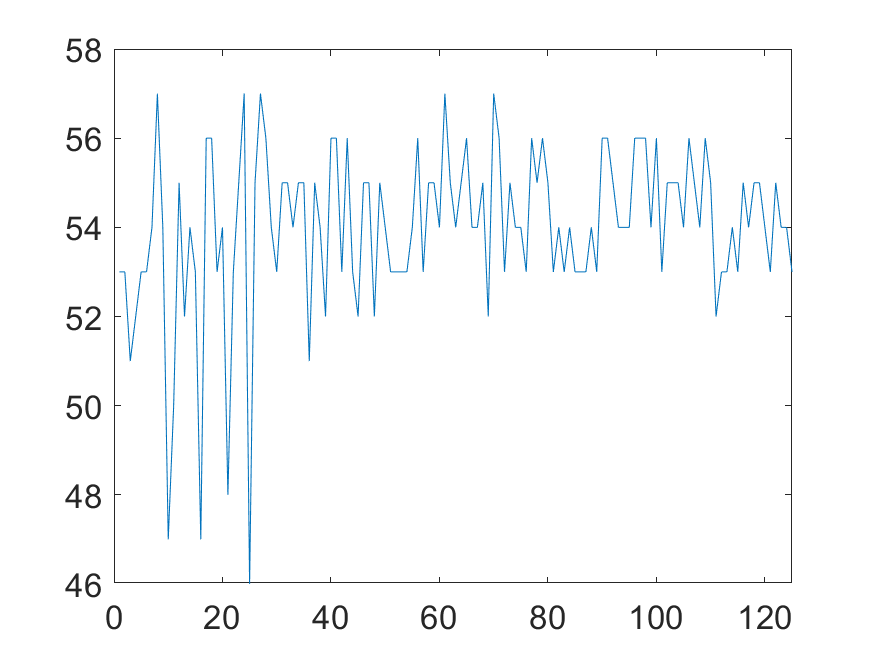}
\caption{Plot of the number of negative eigenvalues for $\omega=50\,rad/s$
($l_p=1.79\,m$, $l_s=0.18\,m$ for the homogenous background
material).}\label{plot_eig_k_50}
\end{figure}
\noindent
\\
The algorithm then compares the calculated positive eigenvalues with the $\tilde{M}$ as stated in the {\bf{Discritization}}.
\\
\\
The calculated values $M_0$ as well as the numerically chosen $\tilde{M}$ for
the frequencies $\omega$ under investigation are given in Table \ref{k_M_0}. We
also added the special case $\omega=0$ (see Figure \ref{plot_eig_k_0}), which
corresponds to the stationary case as considered in \cite{EH21}.
\begin{figure}[H]
\centering 
\includegraphics[width=0.37\textwidth]{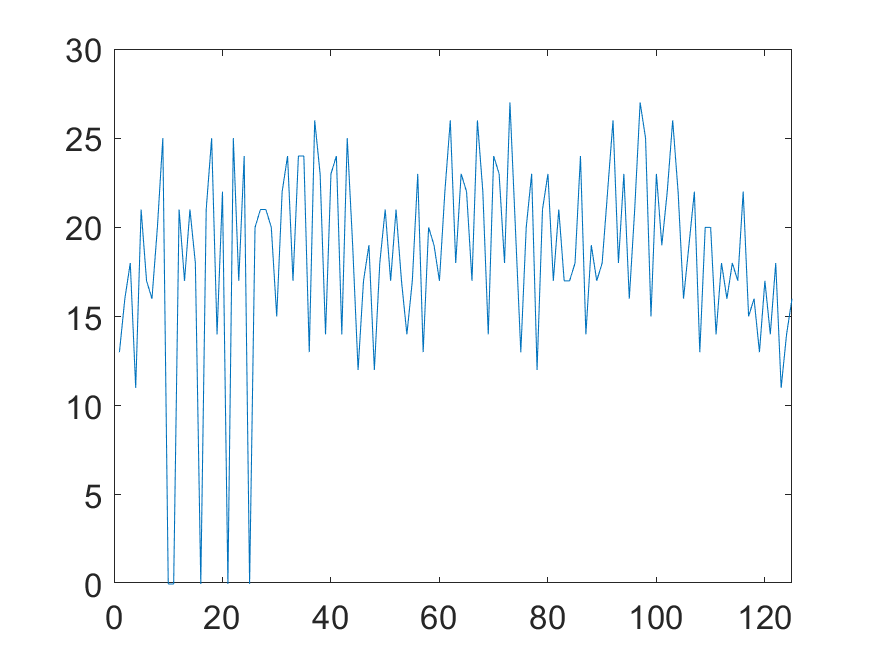}
\caption{Plot of the number of negative eigenvalues for $\omega=0$ for the homogenous background material.}\label{plot_eig_k_0}
\end{figure}
\renewcommand{\arraystretch}{1.3} 
\begin{table} [H]
 \begin{center}
 \begin{tabular}{ |c|c|c|}  
\hline
 $k$ & $M_0$ & $\tilde{M}$\\
  \hline
$\omega =0\,$ & $M_0=0$  & $\tilde{M}=0$\\
 \hline
$\omega =10\, rad/s$ & $M_0=24$ & $\tilde{M}=7$\\
\hline
$\omega =50\, rad/s$ & $M_0=90$  & $\tilde{M}=50$\\
\hline
\end{tabular}
\end{center}
\caption{Calculated values $M_0$ and chosen $\tilde{M}$.}
\label{k_M_0}
\end{table}
\noindent
All in all, applying Algorithm \ref{alg_shapeInclusion} with the corresponding
$\tilde{M}$ given in Table \ref{k_M_0} results in the following reconstructions
for $\omega=0,10, 50 \,rad/s$ (see Figure \ref{reconst_0_10_50}).
\begin{figure}[H]
\centering 
\includegraphics[width=0.4\textwidth]{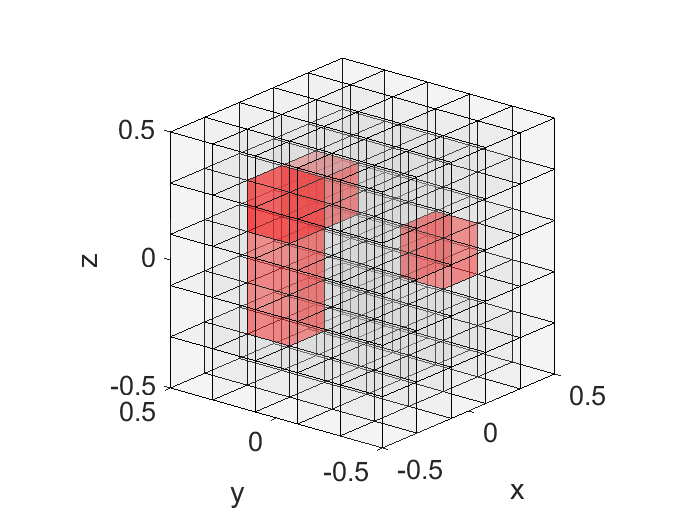}
\caption{Reconstruction for $\omega=0\,rad/s,10\,rad/s, 50\,rad/s$ and the corresponding
$\tilde{M}$ given in Table \ref{k_M_0}.}\label{reconst_0_10_50}
\end{figure}
\noindent
\\
It should be noted that for high frequencies $\omega$, e.g., $\omega=50\,rad/s$, finding a
suitable $\tilde{M}$, which detects the inclusions correctly, gets more and
more difficult. For $k=10$ we could chose $\tilde{M}$ in the range of $4$ to
$12$, but for $k=50$, we can only choose one possible $\tilde{M}$, namely
$\tilde{M}=50$. 
\\
\\
To rectify this, the number of Neumann boundary loads $g_j$ to discretize the operator has to be increased to differentiate the blocks better.
We want to discuss two possible ways to do that. The first is to increase the number of normal boundary loads from $125$ to $500$ and end up with the eigenvalues shown in Figure \ref{plot_eig_k_50_10_x_10}. 
\begin{figure}[H]
\centering 
\includegraphics[width=0.37\textwidth]{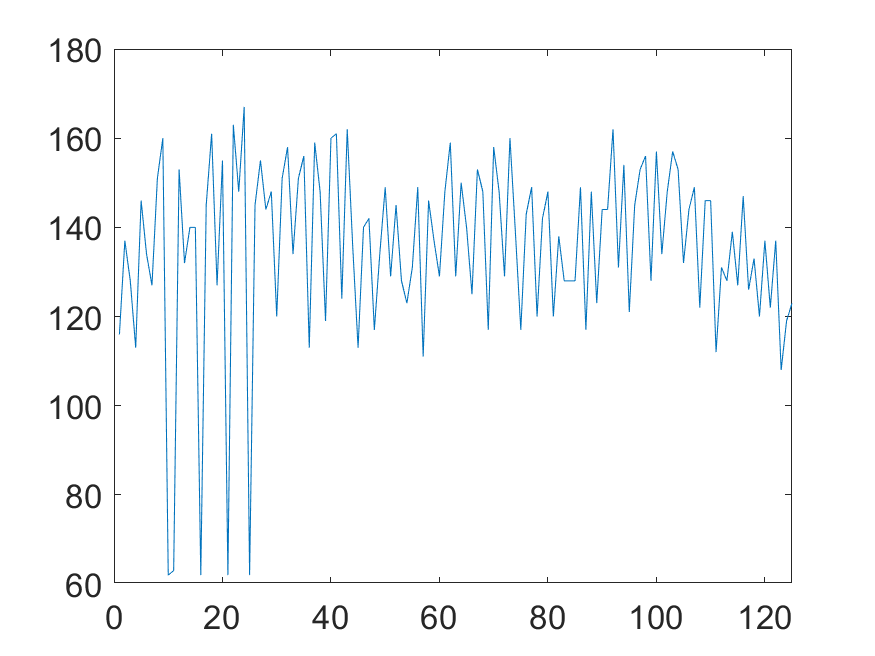}
\caption{Plot of the number of negative eigenvalues for $\omega=50\,rad/s$
($l_p=1.79\,m$, $l_s=0.18\,m$ for the homogenous background material)
and $500$ boundary loads.}\label{plot_eig_k_50_10_x_10}
\end{figure}
\noindent
\\
This allows us to choose $\tilde{M}$ in the range of $63$ to $112$ instead of a single value as we have for the case of $125$ boundary loads. 
\\
\\
Note that we only considered traction in the normal direction in the previous
numerics, which neglects all tangential tractions. One motivation for this is that
that in an practical lab experiment, it is easier to test an object in this
manner. 
We however also considered the case of $125$ normal tractions
plus $250$ tangential tractions resulting in $375$ boundary loads $g_j$, since this 
provides a more complete description of the degrees of freedom.
The results can be seen in Figure \ref{plot_eig_k_50_normal_tangential}. Comparing
Figure \ref{plot_eig_k_50_10_x_10} and Figure
\ref{plot_eig_k_50_normal_tangential}, we see that the distriubution of
negative eigenvalues is nearly the same with Figure \ref{plot_eig_k_50_10_x_10}
covering a slightly wider range. In the case for $375$ boundary loads, we can
choose $\tilde{M}$ in the range of $63$ to $100$.
\begin{figure}[H]
\centering 
\includegraphics[width=0.37\textwidth]{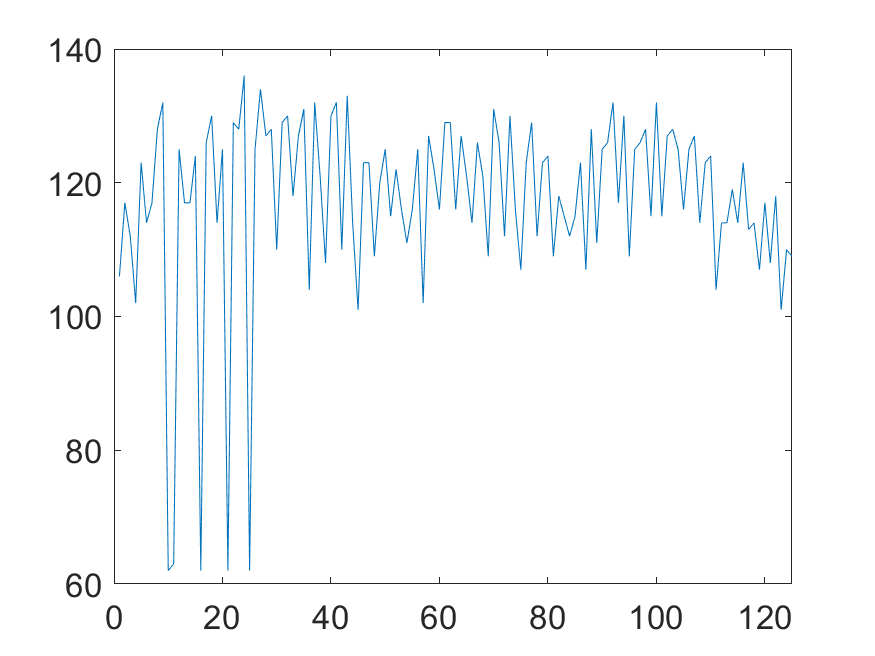}
\caption{Plot of the number of negative eigenvalues for $\omega=50\,rad/s$
($l_p=1.79\,m$, $l_s=0.18\,m$ for the homogenous background material)
and $375$ boundary loads (including $125$ normal and $250$ tangential components).}\label{plot_eig_k_50_normal_tangential}
\end{figure}
\noindent
\\
All in all, taking more boundary loads will stabilize the reconstruction but increases the computation time, respectively.
\\
\\
Finally, we want to remark that our numerical simulations are very preliminary but show that in
principle, the monotonicity tests work. Thus, this will build the basis for further examinations 
and the development of the "linearized monotonicity methods" (see \cite{EH21}) for the stationary case)
which will allow us a faster implementation and as such, a testing with more test inclusions. A further
extension could be the consideration of a monotonicity-based regularization as well as the examination
of the resolution guarantee (similar as considered in
for the stationary elastic inverse problem in \cite{EH22} and \cite{EBH23}, respectively). 
\\
\\
{\bf{Acknowledgement}}
\\
The first author thanks the German Research Foundation (DFG) for funding the project "Inclusion Reconstruction with Monotonicity-based Methods for the Elasto-oscillatory Wave Equation" (reference number 499303971).

\end{document}